\documentclass[11pt]{article}

\usepackage[parfill, skip=.25\baselineskip plus 1pt, indent=17pt]{parskip}

\usepackage[a4paper, margin=1in]{geometry}

\usepackage{amsmath,array}
\usepackage{amssymb}
\usepackage{amsthm}
\usepackage[inline,shortlabels]{enumitem}
\usepackage{graphicx}
\usepackage{hyperref}
\usepackage{times}
\usepackage[cal=boondoxo]{mathalfa}
\usepackage{upgreek}

\allowdisplaybreaks

\setlist[enumerate,1]{label=\roman*., leftmargin=8mm}
\setlist[itemize,1]{label=\textbullet, leftmargin=6mm}

\newtheorem{theorem}{Theorem}[section]
\newtheorem{corollary}{Corollary}[theorem]
\newtheorem{remark}[corollary]{Remark}
\newtheorem{proposition}[theorem]{Proposition}
\newtheorem{lemma}[theorem]{Lemma}
\newtheorem{definition}[theorem]{Definition}

\numberwithin{equation}{section}

\newcommand{\<}{\langle}
\renewcommand{\>}{\rangle}

\DeclareMathOperator{\dist}{dist}
\DeclareMathOperator{\Span}{span}

\newcommand{\norm}[1]{\left\lVert #1\right\rVert}

\newcommand{\brac}[1]{\left( #1\right)}
\newcommand{\sbrac}[1]{\left[ #1\right]}
\newcommand{\inner}[1]{\left\< #1\right\>}

\newcommand{\bs}[1]{\boldsymbol{\mathbf{#1}}}

\newcommand{\eva}[1]{\left.#1\right|}

\newcommand{\ph}{\,\cdot\,}

\newcommand{\N} {\mathbb{N}}

\newcommand{\C} {\mathbb{C}}
\newcommand{\R} {\mathbb{R}}

\newcommand{\lpc}{\Delta}
\newcommand{\grad}{\nabla}

\newcommand{\mb}{\mathbf}

\renewcommand{\d}{\textup{d}}

\title{Nonradial linear stability of liquid Lane--Emden stars}

\author{King Ming Lam\thanks{Delft Institute of Applied Mathematics, Delft University of Technology, 2628 CD Delft, Netherlands. Email: \href{mailto:K.M.Lam@tudelft.nl}{K.M.Lam@tudelft.nl}.}}

\date{}

\begin{document}

\maketitle

\begin{abstract}
The classical model of a star is the Lane--Emden star with dynamics governed by the Euler--Poisson equations. We consider the case of a liquid star with a ``stiffened gas'' equation of state $p=\rho^\gamma-1$. We derive the full 3D linearised Euler--Poisson system around liquid Lane--Emden stars with no symmetry assumptions on the perturbations and show that the associated linear operator $\mb L$ is non-negative whenever the radial mode is non-negative. We show that $\mb L$ has an infinite-dimensional kernel each element of which corresponds to a linearly growing solution to the linearised system. When restricted to irrotational perturbations and modding out the three kernel elements corresponding to momentum conservation, however, we prove that $\mb L$ is strictly positive with coercivity bound $\<\mb L\bs\theta,\bs\theta\>_{\bar\rho}\gtrsim\|\bs\theta\|_{L^2(B_R)}^2$. Hence we demonstrate that the liquid Lane--Emden stars are stable against non-radial irrotational perturbations whenever the purely radial mode is stable, improving upon previous results that dealt only with purely radial perturbations. However, the stability might not be as strong as one might hope, as we prove that $\|\grad\bs\theta\|_{L^2(B_R)}^2$ cannot be controlled even in this case.
\end{abstract}

\section{Introduction}

\subsection{The Euler--Poisson system}

We consider a fundamental model of a self-gravitating compressible fluid, given by the Euler--Poisson system.
The unknowns are the fluid density $\rho\ge0$,  the velocity vector $\mb u$, the fluid pressure $p\ge0$, and the gravitational potential $\phi$.
They solve the system
\begin{alignat}{3}
\partial_t\rho+\grad\cdot(\rho\mb u) &=0
\qquad &&\text{ in } \ \Omega(t), \label{E:CONT}\\
\rho(\partial_t + \mb u\cdot\nabla)\mb u + \grad p + \rho\grad\phi &=\mb 0 
\qquad &&\text{ in } \ \Omega(t), \label{E:MOM}\\
\lpc\phi &=4\pi\rho 
\qquad &&\text{ in } \ \mathbb R^3.\label{E:GRAV}
\end{alignat}
Here the pressure $p$ satisfies the (liquid) polytropic equation of state
\begin{equation}\label{E:EOS}
  p=\rho^\gamma-1\qquad\text{where}\qquad \gamma\in[1,2],
\end{equation}
and the star is isolated, which translates into the asymptotic boundary condition for the gravitational potential:
\begin{align}\label{E:AF}
\lim_{|\mb x|\to\infty}\phi(t,\mb x) &=0.
\end{align}
We refer to the system~\eqref{E:CONT}--\eqref{E:AF} as the liquid (EP)$_{\gamma}$-system.
Moreover $\Omega(t): = \{\mb x:\rho(t,\mb x)>0\}$ is the interior of the support of star density. Note that in~\eqref{E:GRAV} the density is trivially extended by zero to the complement of $\Omega(t)$.

Since we allow the boundary to move, we must complement equations~\eqref{E:CONT}--\eqref{E:AF} with suitable boundary conditions at the vacuum free boundary $\partial\Omega(t)$. We have the following boundary conditions:
\begin{enumerate}[a)]
\item the pressure on the boundary matches that of the vacuum outside, i.e. $p=0$ on $\partial\Omega$;
\item the normal velocity at which the boundary changes is equal to $\mb u\cdot\mb n$ at any point on the boundary, where $\mb n$ denotes the outward unit normal vector to $\partial\Omega$.
\end{enumerate}

\newcommand{\Apar}{A\!\partial}
\newcommand{\cApar}{\A\!\partial}

We shall mostly work in Lagrangian coordinates in this article, as they are particularly well suited to the analysis of fluids featuring a vacuum boundary. Let $\bs\eta(t,\mb x)$ be the fluid flow map, defined through
\[
\partial_t\bs\eta=\mb u\circ\bs\eta\qquad\text{with}\qquad\bs\eta(0,\mb x)=\bs\eta_0(\mb x),
\]
where $\mb u\circ\bs\eta(t,\mb x)=\mb u(t,\bs\eta(t,\mb x))$. The spatial domain is then fixed for all time as $\Omega_0:=\bs\eta_0^{-1}(\Omega(0))$. To reformulate the (EP)$_{\gamma}$-system in the new variables, we introduce
\begin{align*}
\mb v &=\mb u\circ\bs\eta
\hfill\tag{Lagrangian velocity}\\
f &=\rho\circ\bs\eta
\hfill\tag{Lagrangian density}\\
\psi &=\phi\circ\bs\eta
\hfill\tag{Lagrangian potential}\\
A &=(\grad\bs\eta)^{-1}
\hfill\tag{inverse of the deformation tensor}\\
J &=\det(\grad\bs\eta)
\hfill\tag{Jacobian determinant}\\
a &=JA.
\hfill\tag{cofactor matrix of the deformation tensor}
\end{align*}
Under this change of coordinates, the continuity equation becomes $fJ=f_0J_0$ and the momentum equation \eqref{E:MOM} in the domain $\Omega_0$ reads
\begin{equation}\label{E:MOMLAGR}
  \partial_t\mb v + \frac1{f_0J_0}\partial_k(A^k(f_0J_0)^{\gamma} J^{1-\gamma}) + A\grad\psi = \mb 0,
\end{equation}
where Einstein summation convention is used (see Definition \ref{Standard notations}). Moreover, $\psi$ solves the Poisson equation
\begin{equation}
(A\grad)\cdot(A\grad)\psi = 4\pi f_0J_0 J^{-1}.
\end{equation}

\newcommand{\K}{\mathcal{K}}

To better describe the gravitational potential $\psi$, we introduce the following convolutional operators
\begin{subequations}\label{E:KDEF}
\begin{align}
\K &:= 4\pi\lpc^{-1}\mb 1_{B_R}\\
\K_{\partial} &:= 4\pi\lpc^{-1}\delta_{\partial B_R}
\end{align}
\end{subequations}
where $\delta_{\partial B_R}$ is the Dirac measure on $\partial B_R$, i.e.
\begin{subequations}
\begin{align*}
\K f (\mb x) &= - \int_{B_R} \frac{f(\mb y)}{|\mb x-\mb y|}\d \mb y\\
\K_{\partial} f (\mb x) &= - \int_{\partial B_R} \frac{f(\mb y)}{|\mb x-\mb y|}\d S(\mb y).
\end{align*}
\end{subequations}
Note that $\psi$ can be written as
\begin{align*}
\psi(\mb x)
&=(\K\rho)(\bs\eta(\mb x))=-\int{\rho(\mb y)\over|\bs\eta(\mb x)-\mb y|}\d\mb y
=-\int{f(\mb z)J(\mb z)\over|\bs\eta(\mb x)-\bs\eta(\mb z)|}\d\mb z
=-\int{f_0(\mb z)J_0(\mb z)\over|\bs\eta(\mb x)-\bs\eta(\mb z)|}\d\mb z.
\end{align*}

For details of the Lagrangian description of the Euler--Poisson system, we refer to~\cite{Hadzic_Jang_2019}.

\subsection{Lane--Emden stars}\label{Lane--Emden stars def}

We look for time independent spherically symmetric solutions to the Euler--Poisson system of the form $\bs\eta(t,\mb x)=\mb x$. Under this ansatz,
\begin{align*}
\psi(\mb x)
&=-\int{f_0(\mb z)\over|\mb x-\mb z|}\d\mb z
=\K f_0(\mb x)
\end{align*}
and the momentum equation~\eqref{E:MOMLAGR} reduces to
\begin{align}\label{LE ODE pre}
{1\over f_0}\grad(f_0^{\gamma}) + \grad\mathcal Kf_0 = \mb 0,
\end{align}
or equivalently
\begin{align}\label{LE ODE}
0 = 
\begin{cases}
{\gamma\over\gamma-1}\lpc f_0^{\gamma-1} + 4\pi f_0 &\quad\text{ when }\quad\gamma>1\\
\lpc\ln f_0 + 4\pi f_0 &\quad\text{ when }\quad\gamma=1
\end{cases}
\end{align}
plus the condition that $\grad f_0(\mb 0)=\mb 0$ (this is the ``integration constant'' which can be deduced from \eqref{LE ODE pre}). In spherical symmetry we have $\lpc = r^{-2}\partial_r(r^{2}\partial_r)$, this gives the ODE which defines the Lane--Emden stars. One can show that given $f_0(0)$, there exists a unique decreasing solution to this ODE that goes to zero at some finite or infinite $r$. Let $R$ be the point at which $f_0(R)=1$. We write $\bar\rho$ to be the density profile of such liquid Lane--Emden (LE) stars:
\begin{subequations}
\begin{align}
\bar\rho(r) &:= f_0(r)\mb 1_{[0,R]}\\
\bar\rho(\mb x) &:= f_0(|\mb x|)\mb 1_{\bar B_R}.
\end{align}
\end{subequations}

Below is a lemma giving bounds for the derivative of the Lane--Emden density profile which we will use in the main text.

\begin{lemma}\label{Lane--Emden profile derivative prop}
Let $\bar\rho$ be the liquid Lane--Emden density profile with adiabatic index $\gamma$. Then $-(\bar\rho^\gamma)'/r\sim 1$, or more precisely,
\begin{align*}
{4\pi\over 3}<-{(\bar\rho^\gamma)'(r)\over r}<{4\pi\over 3}\bar\rho(0)^2.
\end{align*}
\end{lemma}
\begin{proof}
By \eqref{LE ODE} we have $0=\gamma(r^2\bar\rho^{\gamma-2}\bar\rho')'+4\pi r^2\bar\rho$ and $\bar\rho'(0)=0$, or equivalently
\begin{align*}
-{(\bar\rho^\gamma)'(r)\over r} &={4\pi\bar\rho(r)\over r^3}\int_0^r y^2\bar\rho(y)\d y.
\end{align*}
Using that $1\leq\bar\rho\leq\bar\rho(0)$ we have
\begin{align*}
-{(\bar\rho^\gamma)'(r)\over r} &<{4\pi\bar\rho(0)\over r^3}\int_0^r y^2\bar\rho(0)\d y
={4\pi\over 3}\bar\rho(0)^2\\
-{(\bar\rho^\gamma)'(r)\over r} &>{4\pi\over r^3}\int_0^r y^2\d y
={4\pi\over 3}.\qedhere
\end{align*}
\end{proof}

For details of liquid Lane--Emden stars, we refer to~\cite{Lam_2024}. In this paper we will be studying the stability of such liquid Lane--Emden stars.

\subsection{History and Background}

At the most basic level, a star is a lump of fluid (usually hot gas) in a vacuum acted on by two competing self-generated forces --- its own pressure which wants to expand the gas into the vacuum, and its own gravity which wants to concentrate the gas. Hence the study of stellar dynamics begins with the study of fluid equations.

The Euler equations describing the motion of adiabatic inviscid fluids were first described by Euler in 1757 \cite{Euler_1757}, but even today remain an active area of research. Even in the absence of a vacuum, a free boundary, and self-gravity, it is a highly non-trivial problem. In particular, the Euler equations are prone to singularity/shock formation even from quite regular data, see for example \cite{Christodoulou_2007} by Christodoulou and \cite{Sideris_1985} by Sideris. This makes the study of existence and behaviour and global solutions non-trivial. In the presence of a vacuum and a free boundary, the problem is even trickier due to the boundary movement and the degeneracy or discontinuity across the boundary, which means the classical theory for local existence of hyperbolic systems \cite{Majda_1984} does not apply. Local existence theory for the Euler equations in this setting has only been established recently: for the gaseous case by Coutand and Shkoller \cite{Coutand_Shkoller_2012}, Jang and Masmoudi \cite{Jang_Masmoudi_2015} and Ifrim and Tataru \cite{Ifrim_Tataru_2024}; for the liquid case by Lindblad \cite{Lindblad_2005}, Trakhinin \cite{Trakhinin_2009} and Coutand, Hole and Shkoller \cite{Coutand_Hole_Shkoller_2013}. To make a model for stars, however, one has to couple the Euler equations with self-gravity given by the Poisson equation, the result of which is an additional concentrating dynamics that could potentially result in a blow up. Indeed, it has been shown that solutions could blow up in a finite time \cite{Deng_Xiang_Yang_2003, Makino_1992}. And indeed in our universe, we observe both stable stars like our Sun and stars collapsing into black holes. The study of such rich dynamics has been the occupation of generations of scholars.

Beginning with the foundational works of J. Homer Lane \cite{Lane_1870}, August Ritter, and Robert Emden \cite{Emden_1907} in the late 19th and early 20th centuries, and later expanded by astrophysicists such as Chandrasekhar \cite{Chandrasekhar_1939}, Shapiro and Teukolsky \cite{Shapiro_Teukolsky_1983}, the self-gravitating sphere of fluid has served as the bedrock of astrophysical modeling. In this classical framework, the star is treated as a compressible gas governed by a polytropic equation of state $p = \rho^\gamma$ with motion governed by the Euler--Poisson equations. The \emph{Lane--Emden stars} are spherically symmetric time-independent solutions to the system, representing stars in hydrostatic equilibrium when pressure and gravity balance exactly. Despite and because of its simplicity, it is a good approximation for certain regions of various types of stars, and so it serves as the classical model of stars. At the mass critical index $\gamma=4/3$, there exists another class of spherically symmetric but expanding and collapsing solutions known as the Goldreich--Weber stars that were first described by astrophysicists Goldreich and Weber \cite{Goldreich_Weber_1980} which serve as models for stellar collapse and expansion such as supernova expansion.

Given their physical relevance, it is therefore important to have a local well-posedness theory for the Euler--Poisson system in this free boundary context to ensure the model is meaningful; and secondly, to establish stability of the Lane--Emden stars and Goldreich--Weber stars under perturbation to ensure they are generic objects that can exist in our universe --- a universe with many things that can disturb each other. On the question of local existence, this was done by Gu and Lei \cite{Gu_Lei_2016} and by Had\v{z}i\'c and Jang \cite{Hadzic_Jang_2019}. On the question of stability it was shown \cite{Lin_1997, Jang_Makino_2020} that the gaseous Lane--Emden stars are linearly stable when $4/3\leq\gamma<2$ but unstable when $1<\gamma<4/3$. In the critical case $\gamma=4/3$ the Lane--Emden stars are in fact nonlinearly unstable as the Goldreich--Weber stars exist on the same parameter family as them and hence can be reached by an arbitrarily small perturbation. Recently, nonlinear instability in the range $6/5\leq\gamma<4/3$ has been proven in \cite{Jang_2008, Jang_2014} by Jang, and the dynamics of expanding solutions near these unstable stars have now been detailed by Cheng, Cheng, and Lin \cite{Cheng_Cheng_Lin_2025}. Nonlinear stability in the range $4/3<\gamma<2$ has not been fully proven, but conditioned upon the assumption that a global-in-time solution exists it has been shown by Rein \cite{Rein_2003} and Luo and Smoller \cite{Luo_Smoller_2009} using variational arguments; and more recently this was improved by Lin, Wang, and Zhu \cite{Lin_Wang_Zhu_2024} which in particular proved the nonlinear stability unconditionally in the case of spherical symmetric perturbations. Nonlinear stability of the expanding Goldreich--Weber stars against radially symmetric perturbations was proven by Had\v{z}i\'c and Jang \cite{Hadzic_Jang_2018}, and recently generalised to non-radial perturbations by the author together with Had\v{z}i\'c and Jang in \cite{Hadzic_Jang_Lam}.

However, there are cases where this classical model fails to be adequate. The stability of the classical gaseous Lane--Emden stars depends on (with some caveat as explained later) the type of gas we have (the adiabatic index $\gamma$ represents the ``stiffness'' of the gas, e.g. for helium $\gamma=5/3$) but not the overall mass of the star. Indeed, the natural scaling of the Euler--Poisson equations means that solutions of one mass can be scaled to a solution of a different mass, hence classical Lane--Emden stars of different mass are effectively the same. But in astrophysical reality, due to relativistic effects, there exists a maximum mass for stars beyond which stars would explode or collapse. In particular, a main-sequence star would become unstable and blow mass away upon reaching the Eddington limit \cite{Eddington_1916}; a white dwarf beyond the Chandrasekhar limit \cite{Chandrasekhar_1931} would either collapse into a neutron star or explode as a supernova; a neutron star beyond the Tolman--Oppenheimer--Volkoff (TOV) limit \cite{Tolman_1939, Oppenheimer_Volkoff_1939} would collapse to form a black hole. In practice, main-sequence stars are often still modelled by the classical Lane--Emden stars, but shifting the adiabatic index $\gamma$ to also take into account radiation pressure and not just gas pressure. As mass and hence radiation pressure increase, $\gamma\to 4/3$ and the star becomes unstable. In this way by shifting $\gamma$ to take into account other effects, the classical Lane--Emden stars have also been used to model white dwarfs and neutron stars. However, to properly describe the mechanism of the Chandrasekhar limit and Tolman--Oppenheimer--Volkoff (TOV) limit, a relativistic model is better suited. The Chandrasekhar limit results from the fact that electrons cannot move faster than the speed of light, thus the pressure they create is capped and cannot be infinitely scaled up; the Tolman--Oppenheimer--Volkoff (TOV) limit results from the fact that energy and pressure themselves, and not just mass, also create gravity in general relativity.

A relativistic star model for this is given by the Einstein--Euler system. Had\v{z}i\'c, Lin and Rein \cite{Hadzic_Lin_Rein_2021} have proven the radial linear stability of the relativistic star with small mass and the instability of that with large mass; and Had\v{z}i\'c and Lin have further characterised the transition point \cite{Hadzic_Lin_2021}. In a previous work \cite{Lam_2024}, the author has proven that within the framework of the classical Lane--Emden Euler--Poisson model, if one replaces the equation of state with a ``stiffened gas'' equation of state (effectively adding a constant background pressure, often used to model liquids) $p=\rho^\gamma-1$, the resultant system behaves in a very similar way to the relativistic model --- in the mass supercritical case $1<\gamma<4/3$, the resultant \emph{liquid} Lane--Emden stars are radially linearly stable when it has small mass while unstable when it has large mass (very recently Hao and Miao \cite{Hao_Miao_2024} have proven the instability nonlinearly as well based on the linear theory in \cite{Lam_2024}), in contrast to the classical gaseous Lane--Emden stars which are unstable irrespective of mass. In both the relativistic and liquid case, the effect comes from scale breaking in the pressure -- in the relativistic case, the speed of light as the upper limit of speed imposes a fixed pressure scale; and in the liquid case, the addition of the constant created the same. 

In effect, the liquid equation of state gives an additional stabilising effect that makes small-mass liquid stars stable even in the supercritical range. In a famous scientific debate with Arthur Eddington in the 1920s, Sir James Jeans suggested that liquid stars are more stable than gaseous stars \cite{Jeans_1927, Jeans_1928, Eddington_1928}. Even though his specific theory of main-sequence stars being liquid undergoing fission was shown to be wrong, \cite{Lam_2024} and our work here have shown that there is some truth in his suggestion that liquid stars are more stable than gaseous stars and that indeed in some ways real relativistic stars resemble liquid stars. Moreover, since the core of neutron stars behave as a superfluid, the liquid model we consider could serve as a useful model and bridge towards the study of relativistic stars. In particular, our result here of non-radial linear stability of liquid Lane--Emden stars could serve as a starting point for the investigation of non-radial stability of relativistic stars. Solid foundations have already been laid --- local well-posedness for the liquid Euler--Poisson system as considered in this paper has been proven by Ginsberg, Lindblad and Luo \cite{Ginsberg_Lindblad_Luo_2020}. For the relativistic Euler equations (no self-gravity), it has been done by Oliynyk \cite{Oliynyk} and Miao, Shahshahani and Wu \cite{Miao_Shahshahani_Wu_2021} in the liquid case and Disconzi, Ifrim and Tataru \cite{Disconzi_Ifrim_Tataru_2022} in the gaseous case, but that for the Einstein--Euler system remains an open problem.

\subsection{Notations}

We collect here some notation we will use for the reader's convenience.

\begin{definition}[Notations]\label{Standard notations}
\begin{enumerate}[before=\leavevmode]
\item $\mb e_i$ ($i=1,2,3$) denotes the standard basis of $\R^3$, while $\mb e_r$ denotes the radial unit vector $\mb x/|\mb x|$.
\item We denote $\<\bs\theta_1,\bs\theta_2\>_{\bar\rho} =\<\bs\theta_1,\bar\rho\bs\theta_2\>_{L^2(B_R)}$ and similarly for the norm $\|\ph\|_{\bar\rho}$, i.e. the $L^2$ inner product and norm weighted by $\bar\rho$.
\item We denote by $Y_{lm}$ the real spherical harmonics, see Appendix \ref{A:SPHERICALHARMONICS}.
\end{enumerate}
\end{definition}

\section{Formulation and results}

\subsection{Perturbation of liquid Lane--Emden stars}

Before formulating the stability problem, we must first make use of the labelling gauge freedom and fix the choice of $f_0J_0$ for the general perturbation to be exactly identical to the LE profile, i.e. we set
\begin{align}\label{E:GAUGECHOICE}
f_0J_0 = \bar\rho \ \qquad \text{ on }\ B_R(\mb 0).
\end{align}
This equation can be re-written in the form $\rho_0\circ\bs\eta_0 \det[\grad\bs\eta_0] = \bar\rho$ on the initial domain $B_R(\mb 0)$. By a result of Dacorogna--Moser~\cite{Dacorogna_Moser_1990} and similarly to~\cite{Hadzic_Jang_2018, Hadzic_Jang_2018-2} there exists a choice of an initial bijective map $\bs\eta_0:B_R(\mb 0)\to\Omega(0)$ so that~\eqref{E:GAUGECHOICE} holds true. The gauge fixing condition~\eqref{E:GAUGECHOICE} is necessary as it constrains the freedom to arbitrarily relabel the particles at the initial time.

\begin{lemma}[Euler--Poisson around liquid LE profile]
With respect to the liquid LE profile $\bar\rho$ from section \ref{Lane--Emden stars def}, the perturbation variable
\begin{equation}\label{E:THETADEF}
  \bs\theta(\mb x,t) := \bs\eta(\mb x,t) - \mb x,
\end{equation}
which measures the deviation of the nonlinear flow to the background LE profile, formally solves
\begin{subequations}\label{E:EP around LE}
\begin{align}
\partial_t^2\bs\theta + \bar\rho^{-1}\partial_k(\bar \rho^\gamma(A^kJ^{1-\gamma}-I^k)) + A\grad\psi - \grad\mathcal{K}\bar\rho &= \mb 0\qquad\text{ on }\  B_R\\
J &= 1\qquad\text{ on }\ \partial B_R.
\end{align}
\end{subequations}
\end{lemma}
\begin{proof}
Recall from \eqref{LE ODE} that the LE profile satisfies
\begin{align}
\mb 0={\gamma\over\gamma-1}\grad\bar\rho^{\gamma-1}+\grad\K\bar\rho.
\label{E:equation for bar-w}
\end{align}
Using the gauge condition~\eqref{E:GAUGECHOICE},
the momentum equation~\eqref{E:MOMLAGR} minus the above line becomes
\begin{align*}
\partial_t^2\bs\theta+\bar\rho^{-1}\partial_k(\bar \rho^\gamma(A^kJ^{1-\gamma}-I^k))+A\grad\psi-\grad\mathcal{K}\bar\rho=\mb 0.
\end{align*}
Moreover, we have $J=fJ=f_0J_0=\bar\rho=1$ on $\partial B_R$.
\end{proof}

Next we derive the linearisation of the system.

\begin{lemma}[Linearised Euler--Poisson around liquid LE profile]\label{L:LINEAREP}
The formal linearisation of \eqref{E:EP around LE} reads
\begin{subequations}\label{E:LINDYN}
\begin{align}
\partial_t^2\bs\theta+\mb L\bs\theta &= \mb 0\qquad\text{ on }\ B_R\\
\grad\cdot\bs\theta&=0\qquad\text{ on }\ \partial B_R
\end{align}
\end{subequations}
where
\begin{align}\label{E:LINDEF}
\mb L\bs\theta&:=-\grad\brac{(\gamma\bar\rho^{\gamma-2}+\K)\grad\cdot(\bar\rho\bs\theta)+\int_{\partial{B_R}}{\bs\theta(\mb z)\over|\mb x-\mb z|}\cdot\d\mb S(\mb z)}\nonumber\\
&\ =-\grad\brac{(\gamma\bar\rho^{\gamma-2}+\K-\bar\rho'(R)^{-1}\K_{\partial})\grad\cdot(\bar\rho\bs\theta)}.
\end{align}
\end{lemma}
\begin{proof}
Since $\grad\bs\eta=I+\grad\bs\theta$, to first order (in $\bs\theta$) we have $A=I-\grad\bs\theta$ and $J=1+\grad\cdot\bs\theta$. So to first order we have
\begin{align*}
AJ^{1-\gamma}&=(I-\grad\bs\theta)(1+\grad\cdot\bs\theta)^{1-\gamma}
=(I-\grad\bs\theta)\brac{1-(\gamma-1)\grad\cdot\bs\theta}
=\brac{1-(\gamma-1)\grad\cdot\bs\theta}I-\grad\bs\theta
\end{align*}
and
\begin{align*}
{1\over\bar\rho}\partial_k(\bar\rho^\gamma(A^kJ^{1-\gamma}-I^k))
&=-{\gamma-1\over\bar\rho}\grad(\bar\rho^\gamma\grad\cdot\bs\theta)-{1\over\bar\rho}\partial_k(\bar\rho^\gamma\grad\theta^k)\\
&=-\gamma\grad(\bar\rho^{\gamma-1}\grad\cdot\bs\theta)-{\gamma\over\gamma-1}\grad(\bs\theta\cdot\grad\bar\rho^{\gamma-1})+{\gamma\over\gamma-1}\bs\theta\cdot\grad\grad\bar\rho^{\gamma-1}\\
&=-\gamma\grad(\bar\rho^{\gamma-2}\grad\cdot(\bar\rho\bs\theta))+{\gamma\over\gamma-1}\bs\theta\cdot\grad\grad\bar\rho^{\gamma-1}\\
&=-\gamma\grad(\bar\rho^{\gamma-2}\grad\cdot(\bar\rho\bs\theta))-\bs\theta\cdot\grad\grad\K\bar\rho.
\end{align*}
Since
\begin{align*}
|\bs\eta(\mb x)-\bs\eta(\mb z)|^2 
&=|\mb x-\mb z+\bs\theta(\mb x)-\bs\theta(\mb z)|^2
=|\mb x-\mb z|^2+2(\mb x-\mb z)\cdot(\bs\theta(\mb x)-\bs\theta(\mb z))+|\bs\theta(\mb x)-\bs\theta(\mb z)|^2,
\end{align*}
to first order we have
\[
{1\over|\bs\eta(\mb x)-\bs\eta(\mb z)|}={1\over|\mb x-\mb z|}\brac{1-{(\mb x-\mb z)\cdot(\bs\theta(\mb x)-\bs\theta(\mb z))\over|\mb x-\mb z|^2}}.
\]
Consequently, to first order, we obtain
\begin{align*}
\psi(\mb x)&=-\int{\bar\rho(\mb z)\over|\bs\eta(\mb x)-\bs\eta(\mb z)|}\d\mb z
=-\int{\bar\rho(\mb z)\over|\mb x-\mb z|}\d\mb z+\int{(\mb x-\mb z)\cdot(\bs\theta(\mb x)-\bs\theta(\mb z))\over|\mb x-\mb z|^3}\bar\rho(\mb z)\d\mb z\\
&=(\K\bar\rho)(\mb x)+\int\brac{-\bs\theta(\mb x)\cdot\grad_{\mb x}{1\over|\mb x-\mb z|}-\bs\theta(\mb z)\cdot\grad_{\mb z}{1\over|\mb x-\mb z|}}\bar\rho(\mb z)\d\mb z\\
&=(\K\bar\rho)(\mb x)+\bs\theta\cdot\grad(\K\bar\rho)(\mb x)-(\K\grad\cdot(\bar\rho\bs\theta))(\mb x)-\int_{\partial{B_R}}{\bs\theta(\mb z)\over|\mb x-\mb z|}\cdot\d\mb S(\mb z)
\end{align*}
and
\begin{align*}
A\grad\psi&=(I^i-\grad\theta^i)\partial_i\brac{\K\bar\rho+\bs\theta\cdot\grad\K\bar\rho-\K\grad\cdot(\bar\rho\bs\theta)-\int_{\partial{B_R}}{\bs\theta(\mb z)\over|\mb x-\mb z|}\cdot\d\mb S(\mb z)}\\
&=\grad\K\bar\rho-(\grad\theta^i)\partial_i\K\bar\rho+\grad(\bs\theta\cdot\grad\K\bar\rho)-\grad\K\grad\cdot(\bar\rho\bs\theta)-\grad\int_{\partial{B_R}}{\bs\theta(\mb z)\over|\mb x-\mb z|}\cdot\d\mb S(\mb z)\\
&=\grad\K\bar\rho+\bs\theta\cdot\grad\grad\K\bar\rho-\grad\K\grad\cdot(\bar\rho\bs\theta)-\grad\int_{\partial{B_R}}{\bs\theta(\mb z)\over|\mb x-\mb z|}\cdot\d\mb S(\mb z).
\end{align*}
Therefore the linearisation of the momentum equation~\eqref{E:EP around LE} takes the form~\eqref{E:LINDYN}.
\end{proof}

\subsection{Properties of $\mb L$}

Studying the linear stability of the liquid LE profile entails analysing the linear operator $\mb L$. In particular we need a precise understanding of its coercivity properties. Hence the subject of this section is the properties of $\mb L$.

Before stating the main theorem, we first establish some basic properties of $\mb L$.

\begin{lemma}[Symmetry of $\mb L$]
$\mb L$ is symmetric in the space $H=\{\bs\theta\in H^2(B_R):(\grad\cdot\bs\theta)|_{\partial B_R}=0\}$ and
\begin{align*}
\<\mb L\bs\theta_1,\bs\theta_2\>_{\bar\rho}
&=\int_{B_R}\grad\cdot(\bar\rho\bs\theta_2)\brac{\gamma\bar\rho^{\gamma-2}+\K-\bar\rho'(R)^{-1}\K_{\partial}}\grad\cdot(\bar\rho\bs\theta_1)\d\mb x\\
&\quad-\int_{\partial B_R}\bar\rho'(R)^{-1}\grad\cdot(\bar\rho\bs\theta_2)\brac{\gamma+\K-\bar\rho'(R)^{-1}\K_{\partial}}\grad\cdot(\bar\rho\bs\theta_1)\d S.
\end{align*}
\end{lemma}
\begin{proof}
Since $(\grad\cdot\bs\theta_2)|_{\partial B_R}=0$, we have $\grad\cdot(\bar\rho\bs\theta_2)|_{\partial B_R}=\bar\rho'(R)\bs\theta_2\cdot\mb e_r$. So we have
\begin{align*}
\<\mb L\bs\theta_1,\bs\theta_2\>_{\bar\rho}&=\int_{B_R}\grad\cdot(\bar\rho\bs\theta_2)\brac{\gamma\bar\rho^{\gamma-2}+\K-\bar\rho'(R)^{-1}\K_{\partial}}\grad\cdot(\bar\rho\bs\theta_1)\d\mb x\\
&\quad-\int_{\partial B_R}\brac{\gamma+\K-\bar\rho'(R)^{-1}\K_{\partial}}\grad\cdot(\bar\rho\bs\theta_1)\bs\theta_2\cdot\d\mb S\\
&=\int_{B_R}\grad\cdot(\bar\rho\bs\theta_2)\brac{\gamma\bar\rho^{\gamma-2}+\K-\bar\rho'(R)^{-1}\K_{\partial}}\grad\cdot(\bar\rho\bs\theta_1)\d\mb x\\
&\quad-\int_{\partial B_R}\bar\rho'(R)^{-1}\grad\cdot(\bar\rho\bs\theta_2)\brac{\gamma+\K-\bar\rho'(R)^{-1}\K_{\partial}}\grad\cdot(\bar\rho\bs\theta_1)\d S.
\end{align*}
Note this is defined for $\bs\theta_i$ ($i=1,2$) such that $\grad\cdot(\bar\rho\bs\theta_i)\in H^1(B_R)$. We see that $\mb L$ is symmetric under $\<\ph,\ph\>_{\bar\rho}$ since
\begin{align*}
\int_{B_R}\grad\cdot(\bar\rho\bs\theta_2)\K\grad\cdot(\bar\rho\bs\theta_1)\d\mb x
&=-\int_{B_R}\int_{B_R}{\grad\cdot(\bar\rho\bs\theta_2)(\mb x)\grad\cdot(\bar\rho\bs\theta_1)(\mb y)\over|\mb x-\mb y|}\d\mb x\d\mb y\\
\int_{\partial B_R}\grad\cdot(\bar\rho\bs\theta_2)\K_{\partial}\grad\cdot(\bar\rho\bs\theta_1)\d S
&=-\int_{\partial B_R}\int_{\partial B_R}{\grad\cdot(\bar\rho\bs\theta_2)(\mb x)\grad\cdot(\bar\rho\bs\theta_1)(\mb y)\over|\mb x-\mb y|}\d S(\mb x)\d S(\mb y)\\
\int_{B_R}\grad\cdot(\bar\rho\bs\theta_2)\K_{\partial}\grad\cdot(\bar\rho\bs\theta_1)\d\mb x
&=-\int_{B_R}\int_{\partial B_R}{\grad\cdot(\bar\rho\bs\theta_1)(\mb y)\grad\cdot(\bar\rho\bs\theta_2)(\mb x)\over|\mb x-\mb y|}\d S(\mb y)\d\mb x\\
\int_{\partial B_R}\grad\cdot(\bar\rho\bs\theta_2)\K\grad\cdot(\bar\rho\bs\theta_1)\d S
&=-\int_{B_R}\int_{\partial B_R}{\grad\cdot(\bar\rho\bs\theta_2)(\mb y)\grad\cdot(\bar\rho\bs\theta_1)(\mb x)\over|\mb x-\mb y|}\d S(\mb y)\d\mb x.\qedhere
\end{align*}
\end{proof}

\begin{remark}
Note that for solutions $\bs\theta$ of \eqref{E:LINDYN}, using the symmetry of $\mb L$, we have
\begin{align*}
0=\<\partial_t^2\bs\theta+\mb L\bs\theta,\partial_t\bs\theta\>_{\bar\rho}={1\over 2}\partial_t\brac{\<\partial_t\bs\theta,\partial_t\bs\theta\>_{\bar\rho}+\<\mb L\bs\theta,\bs\theta\>_{\bar\rho}}.
\end{align*}
Therefore, $\|\partial_t\bs\theta\|_{\bar\rho}^2+\<\mb L\bs\theta,\bs\theta\>_{\bar\rho}$ is a conserved quantity. If $\mb L$ is coercive in the sense that $\<\mb L\bs\theta,\bs\theta\>_{\bar\rho}\gtrsim\|\bs\theta\|_{H^s(B_R)}^2$, this means the norm of the perturbation $\|\bs\theta\|_{H^s(B_R)}$ stays bounded for all time, hence stability. Therefore the study of stability of liquid LE stars boils down to the study of the positivity of $\mb L$.
\end{remark}

\begin{lemma}[$\K,\grad$ commutators]\label{commutators}
$[\K,\grad]=\K_{\partial}\mb e_r$.
\end{lemma}
\begin{proof}
We have
\begin{align*}
(\K\grad h)(\mb x)&=-\int_{B_R}{\grad h(\mb y)\over|\mb x-\mb y|}\d\mb y
=\int_{B_R}h(\mb y)\grad_{\mb y}{1\over|\mb x-\mb y|}\d\mb y-\int_{\partial B_R}{h(\mb y)\over|\mb x-\mb y|}\d\mb S(\mb y)\\
&=-\int_{B_R}h(\mb y)\grad_{\mb x}{1\over|\mb x-\mb y|}\d\mb y-\int_{\partial B_R}{h(\mb y)\over|\mb x-\mb y|}\d\mb S(\mb y)\\
&=(\grad\K h)(\mb x)+\K_{\partial}(\mb e_rh).\qedhere
\end{align*}
\end{proof}

\begin{proposition}[Translational eigenfunctions]\label{Eigenfunctions for L}
$\mb e_i$ are eigenfunctions for $\mb L$ with eigenvalue 0.
\end{proposition}
\begin{proof}
Let $\mb f\in\mathbb R^3$ be a constant vector. Since $\mb 0={\gamma\over \gamma-1}\grad\bar\rho^{\gamma-1}+\grad\K\bar\rho$ \eqref{LE ODE} and $[\K,\grad]\bar\rho=\K_{\partial}\mb e_r\bar\rho=\K_{\partial}\mb e_r$ (Lemma \ref{commutators}), we have
\begin{align*}
\mb L\mb f &=-\grad\brac{(\gamma\bar\rho^{\gamma-2}+\K-\bar\rho'(R)^{-1}\K_{\partial})\grad\cdot(\bar\rho\mb f)}\\
&=-\grad\brac{\mb f\cdot(\gamma\bar\rho^{\gamma-2}+\K-\bar\rho'(R)^{-1}\K_{\partial})\grad\bar\rho}\\
&=-\grad\brac{\mb f\cdot(\K_{\partial}\mb e_r-\bar\rho'(R)^{-1}\K_{\partial}\grad\bar\rho)}
=\mb 0.\qedhere
\end{align*}
\end{proof}

\begin{remark}\label{Translational stability}
  These eigenfunctions correspond to the Galilean invariance of our system and the three directions of momentum conservation. For any constant vector $\mb f$, we see that $\bs\theta=t\mb f$ is a solution to our systems \eqref{E:LINDYN} and \eqref{E:EP around LE}. This is just the LE star translated by a constant velocity in the direction of $\mb f$. So even though here the perturbation $\bs\theta$ is growing in time, it does not represent any real instability of the star.
\end{remark}

The main result of this paper is that $\mb L$ is non-negative; and if the perturbation $\bs\theta$ is orthogonal to the three eigenfunctions from Proposition~\ref{Eigenfunctions for L}, then the operator $\mb L$ has a quantitative lower bound as a function of $g=\grad\cdot(\bar\rho\bs\theta)$.

\begin{theorem}[Boundedness and non-negativity of $\mb L$]\label{linear operator coercivity}
Let $\mb L$ be as defined in \eqref{E:LINDEF} and
\[H=\{\bs\theta\in H^2(B_R):(\grad\cdot\bs\theta)|_{\partial B_R}=0\}.\]
Then for $\bs\theta\in H$ we have
\begin{align}\label{L upper bound}
\<\mb L\bs\theta,\bs\theta\>_{\bar\rho}\lesssim\|\grad\cdot(\bar\rho\bs\theta)\|_{L^2(B_R)}^2+{1\over|\bar\rho'(R)|}\|\grad\cdot(\bar\rho\bs\theta)\|_{L^2(\partial B_R)}^2.
\end{align}
And moreover,
\begin{enumerate}[i.]
\item When either $\gamma\geq 4/3$ or $\bar\rho(0)$ is close to 1, we have
\begin{align*}
\inf_{\|\bs\theta\|_H=1}\<\mb L\bs\theta,\bs\theta\>_{\bar\rho}\geq 0.
\end{align*}
And if in addition $\bs\theta\in H$ is such that $\<\bs\theta,\mb e_i\>_{\bar\rho}=0$ for all $i=1,2,3$, then we have
\begin{align}\label{L lower bound}
\<\mb L\bs\theta,\bs\theta\>_{\bar\rho}\gtrsim\|\grad\cdot(\bar\rho\bs\theta)\|_{L^2(B_R)}^2+{1\over\bar\rho'(R)^2}\|\grad\cdot(\bar\rho\bs\theta)\|_{L^2(\partial B_R)}^2.
\end{align}
\item When $\gamma<4/3$ and $\bar\rho(0)$ is large, we have
\begin{align*}
\inf_{\|\bs\theta\|_H=1}\<\mb L\bs\theta,\bs\theta\>_{\bar\rho}<0.
\end{align*}
\end{enumerate}
\end{theorem}

The proof of this theorem will be given in Section \ref{Proof of Theorem section}. For an outline of the proof and how the method differs from related works, see the beginning of Section \ref{Proof of Theorem section}.

\begin{remark}
\eqref{L lower bound} implies that
\begin{align*}
\<\mb L\bs\theta,\bs\theta\>_{\bar\rho}\gtrsim\|\lpc\Psi\|_{L^2(B_R)}^2+\|\grad\Psi\|_{L^2(\R^3)}^2
\end{align*}
where $\Psi$ is the gravitational potential induced by the flow disturbance $\bar\rho\bs\theta$:
\begin{align*}
\Psi&:={1\over 4\pi}\K\grad\cdot(\bar\rho\bs\theta)\\
\lpc\Psi&\;=\grad\cdot(\bar\rho\bs\theta)=:g.
\end{align*}
This is because, using Hardy's inequality $\|\Psi/|\mb x|\|_{L^2}\leq 2\|\grad\Psi\|_{L^2}$,
\begin{align*}
\|\grad\Psi\|_{L^2(\R^3)}^2
=\<\Psi,g\>_{L^2(B_R)}
\leq\norm{\Psi\over|\mb x|}_{L^2(B_R)}\norm{|\mb x|g}_{L^2(B_R)}
\leq 2\norm{\grad\Psi}_{L^2(\R^3)}R\norm{g}_{L^2(B_R)}.
\end{align*}
Therefore $\|\grad\Psi\|_{L^2(\R^3)}\leq 2R\|g\|_{L^2(B_R)}$.
\end{remark}

\medskip

\begin{corollary}[Kernel of $\mb L$]
In case i.~of theorem \ref{linear operator coercivity}, $\ker\mb L = \Span\{\mb e_1,\mb e_2,\mb e_3\}\cup\{\bs\theta:\grad\cdot(\bar\rho\bs\theta)=0\}$; and in case ii.~we have $\supseteq$ instead of $=$. In particular, $\ker\mb L$ is infinite dimensional and each element of the kernel corresponds to a linearly growing solution to the linearised Euler--Poisson system \ref{E:LINDYN} around liquid LE profile.
\end{corollary}
\begin{proof}
From \eqref{E:LINDEF} and Lemma \ref{Eigenfunctions for L}, we see that $\{\mb e_1,\mb e_2,\mb e_3\}\cup\{\bs\theta:\grad\cdot(\bar\rho\bs\theta)=0\}\subseteq\ker\mb L$. And in case i.~by \eqref{L lower bound} we see that the $\subseteq$ is in fact an equality. For any smooth enough vector field $\mb A$, we see that $\bs\theta=\bar\rho^{-1}\grad\times\mb A$ is such that $\grad\cdot(\bar\rho\bs\theta)=0$. We can pick infinitely many disjoint balls in $B_R$, and for each such ball there exists a non-zero $\mb A$ that is supported entirely in it. This gives us infinitely many elements of the kernel of $\mb L$ with disjoint support and hence are linearly independent. If $\Theta\in\ker\mb L$, then we see that $\bs\theta=t\Theta$ is a solution to \eqref{E:LINDYN}.
\end{proof}

\begin{remark}
This infinite-dimensional kernel is analogous to that found for the gaseous Lane--Emden stars in \cite{Jang_Makino_2020}, which also contain the constant vectors and all vector fields for which the analogous weighted divergence vanish.
\end{remark}

We have seen how the unit vectors $\{\mb e_1,\mb e_2,\mb e_3\}$ do not give rise to genuine instability in remark \ref{Translational stability}. But the $\{\bs\theta:\grad\cdot(\bar\rho\bs\theta)=0\}$ kernel could potentially give rise to unstable perturbations of the non-linear system as well as the linear one. If we restrict to irrotational perturbations, however, then the infinite kernel is modded out and we can get strict coercivity.

\begin{corollary}[Coercivity of $\mb L$]\label{linear operator coercivity cor}
Suppose case i. of Theorem \ref{linear operator coercivity} holds. If $\bs\theta\in H$ is irrotational and $\<\bs\theta,\mb e_i\>_{\bar\rho}=0$ for all $i=1,2,3$, then we have
\begin{align}\label{L lower bound cor}
\<\mb L\bs\theta,\bs\theta\>_{\bar\rho}\gtrsim\|\bs\theta\|_{L^2(B_R)}^2+\|\bs\theta\cdot\mb e_r\|_{L^2(\partial B_R)}^2.
\end{align}
\end{corollary}
\begin{proof}
Using \eqref{L lower bound} and the fact that $(\grad\cdot\bs\theta)|_{\partial B_R}=0$ we have
\begin{align*}
\<\mb L\bs\theta,\bs\theta\>_{\bar\rho}\gtrsim{1\over\bar\rho'(R)^2}\|\grad\cdot(\bar\rho\bs\theta)\|_{L^2(\partial B_R)}^2=\|\bs\theta\cdot\mb e_r\|_{L^2(\partial B_R)}^2.
\end{align*}
Since $\bs\theta$ is irrotational, $\bs\theta=\grad\phi$ for some scalar function $\phi$. We can assume $\phi$ has zero average on $B_R$ by replacing $\phi$ with $\phi-\int_{B_R}\phi\,\d\mb x$. Then by Poincaré inequality and trace we have
\begin{align*}
\|\phi\|_{L^2(B_R)}+\|\phi\|_{L^2(\partial B_R)}\lesssim\|\grad \phi\|_{L^2(B_R)}=\|\bs\theta\|_{L^2(B_R)}.
\end{align*}
Now
\begin{align*}
\|\bs\theta\|_{\bar\rho}^2
&=\int_{B_R}\bar\rho\bs\theta\cdot\grad\phi\,\d\mb x
=\int_{\partial B_R}\phi\bs\theta\cdot\d\mb S-\int_{B_R}\phi\grad\cdot(\bar\rho\bs\theta)\d\mb x\\
&=\int_{\partial B_R}\phi\bar\rho'(R)^{-1}\grad\cdot(\bar\rho\bs\theta)\d S-\int_{B_R}\phi\grad\cdot(\bar\rho\bs\theta)\d\mb x\\
&\lesssim\|\bs\theta\|_{L^2(B_R)}\brac{\|\grad\cdot(\bar\rho\bs\theta)\|_{L^2(B_R)}+{1\over|\bar\rho'(R)|}\|\grad\cdot(\bar\rho\bs\theta)\|_{L^2(\partial B_R)}}\\
&\lesssim\epsilon\|\bs\theta\|_{L^2(B_R)}^2+\epsilon^{-1}\<\mb L\bs\theta,\bs\theta\>_{\bar\rho}
\end{align*}
where in the last line we used Young's inequality and \eqref{L lower bound}. Taking $\epsilon$ small enough and absorbing the first term on the RHS to the LHS we are done.
\end{proof}

However, even if we restrict to such irrotational perturbations, $\|\grad\bs\theta\|_{L^2(B_R)}^2$ cannot be controlled by $\<\mb L\bs\theta,\bs\theta\>_{\bar\rho}$ as the following proposition shows. This suggests such liquid stars are not as stable as one might hope, especially for ``high frequency'' perturbations (corresponding to higher spherical harmonics):

\begin{proposition}[Failure to control derivative]\label{Failure to control derivative}
Take the setting of Corollary \ref{linear operator coercivity cor},
\[\<\mb L\bs\theta,\bs\theta\>_{\bar\rho}\not\gtrsim\|\grad\bs\theta\|_{L^2(B_R)}^2.\]
\end{proposition}
\begin{proof}
Let $\bs\theta_l=\grad(r^lY_{lm})$, then using \eqref{lpc of spherical harmonics} we have $\grad\cdot(\bar\rho\bs\theta_l)=l\bar\rho'r^{l-1}Y_{lm}$. We have
\begin{align*}
\|\grad\bs\theta_l\|_{L^2(B_R)}^2
&=\int_{B_R}\sum_{i,j\in\{1,2,3\}}\partial_i\partial_j(r^lY_{lm})\partial_i\partial_j(r^lY_{lm})\d\mb x\\
&=\int_{\partial B_R}\sum_{i,j\in\{1,2,3\}}\partial_i\partial_j(r^lY_{lm})\partial_j(r^lY_{lm})\mb e_i\cdot\d\mb S(x)\\
&\quad-\int_{\partial B_R}\sum_{i,j\in\{1,2,3\}}\partial_i\partial_i(r^lY_{lm})\partial_j(r^lY_{lm})\mb e_j\cdot\d\mb S(x)\\
&={1\over 2}\int_{\partial B_R}\partial_r|\grad(r^lY_{lm})|^2\d S(x)
={1\over 2}\int_{\partial B_R}\partial_r|lr^{l-1}Y_{lm}\mb e_r+r^l\grad Y_{lm}|^2\d S(x)\\
&=\int_{\partial B_R}(l(l-1)r^{l-2}Y_{lm}\mb e_r+(l-1)r^{l-1}\grad Y_{lm})(lr^{l-1}Y_{lm}\mb e_r+r^l\grad Y_{lm})\d S(x)\\
&=\int_{\partial B_R}(l^2(l-1)r^{2l-3}Y_{lm}^2+(l-1)r^{2l-3}(r\grad Y_{lm})^2)\d S(x)\\
&=l^2(l-1)R^{2l-1}+l(l-1)(l+1)R^{2l-1}
=l(l-1)(2l+1)R^{2l-1}\\
\|\grad\cdot(\bar\rho\bs\theta_l)\|_{L^2(B_R)}^2
&=\int_{B_R}l^2\bar\rho'^2r^{2l-2}Y_{lm}^2\d\mb x\\
&\lesssim\int_{B_R}l^2r^{2l}Y_{lm}^2\d\mb x
=\int_0^Rl^2r^{2l+2}\d r
={l^2\over 2l+3}R^{2l+3}\\
\|\grad\cdot(\bar\rho\bs\theta_l)\|_{L^2(\partial B_R)}^2
&=\int_{\partial B_R}l^2\bar\rho'^2r^{2l-2}Y_{lm}^2\d S(\mb x)
=l^2\bar\rho'(R)^2R^{2l}
\end{align*}
where the inequality is by Lemma \ref{Lane--Emden profile derivative prop}. Therefore, using \eqref{L upper bound}, we have
\begin{align*}
{\|\grad\bs\theta_l\|_{L^2(B_R)}^2\over\<\mb L\bs\theta_l,\bs\theta_l\>_{\bar\rho}}
&\geq{\|\grad\bs\theta_l\|_{L^2(B_R)}^2\over\|\grad\cdot(\bar\rho\bs\theta_l)\|_{L^2(B_R)}^2+|\bar\rho'(R)|^{-1}\|\grad\cdot(\bar\rho\bs\theta_l)\|_{L^2(\partial B_R)}^2}\\
&\gtrsim{l(l-1)(2l+1)R^{2l-1}\over{l^2\over 2l+3}R^{2l+3}+l^2|\bar\rho'(R)|R^{2l}}\to\infty\qquad\text{ as }\qquad l\to\infty.\qedhere
\end{align*}
\end{proof}

\begin{remark}
The control of $\grad\bs\theta$ fails at the boundary, not the interior. By a type of Gaffney inequality or Hodge-type bound (see e.g. Lemma 2.3 in \cite{Hadzic_Jang_Lam}), we do have $\<\mb L\bs\theta,\bs\theta\>_{\bar\rho}\gtrsim\|\dist(\bs\theta,\partial B_R)\grad\bs\theta\|_{L^2(B_R)}^2$. We expect the addition of surface tension to the system would give us control of $\|\grad\bs\theta\|_{L^2(B_R)}^2$.
\end{remark}

\section{Proof of Theorem \ref{linear operator coercivity}}\label{Proof of Theorem section}

The first important step in making the problem tractable for analysis was already done in Lemma \ref{L:LINEAREP} where the linear liquid boundary condition $(\grad\cdot\bs\theta)|_{\partial B_R}=0$ was derived and crucially used to re-write $\mb L\bs\theta$ entirely as a function of the scalar field $g=\grad\cdot(\bar\rho\bs\theta)$ rather than the full vector field $\bs\theta$. In particular, we have rewritten the liquid boundary contribution $\bs\theta\cdot\mb e_r$ as $\bar\rho'^{-1}g$ on $\partial B_R$. This boundary contribution is unique to the liquid case and is absent in the gaseous case. Rewriting the liquid boundary terms in this way came at a cost however -- in terms of $g$, the boundary terms formally look to be of higher order in derivative count than in the original variable, which makes the analysis tricky in a later stage.

Having written the problem in terms of $g$, we next use spherical harmonics to break down the problem into a sequence of scalar problems for each individual mode $g_{lm}$ (where $l\in\N_0$ and $m\in\{-l,\cdots,l\}$), by analogy to the work of Jang and Makino on gaseous Lane--Emden stars \cite{Jang_Makino_2020}. For each mode we get a bulk contribution $\Lambda_{lm}$ and a boundary contribution $\Gamma_{lm}$, the sum of which we have to show is positive. The radial $l=0$ mode translates exactly to the case of radial stability as studied in \cite{Lam_2024}. In the gaseous case of \cite{Jang_Makino_2020} only the bulk contribution $\Lambda_{lm}$ exists, which they show is positive. Even though the same method of \cite{Jang_Makino_2020} can be used to show that the $\Lambda_{lm}$ in our liquid case is positive, it is not strong enough to show that $\Lambda_{lm}+\Gamma_{lm}$ is positive, at least not for the modes $l=2,3,4,5$. The problem is twofold. Firstly, \cite{Jang_Makino_2020} approximated $\Lambda_{lm}$ from below by an elliptic operator for the case $l\geq 1$, but this approximation produces loss as the elliptic operator is less positive than $\Lambda_{lm}$ itself. Secondly, our liquid boundary terms $\Gamma_{lm}$ do not have a clear sign and, as aforementioned, they are formally of higher order than the bulk terms $\Lambda_{lm}$, which means the excess positivity of $\Lambda_{lm}$ cannot be easily transferred to compensate for $\Gamma_{lm}$.

To get around this, our innovation here is to re-write $g_{lm}$ in terms of the new variable $\chi_{lm}$ (Lemma \ref{chi_lm formulation}) under which the positivity of our problem is naturally apparent. Indeed, in this new variable, we can write $\Lambda_{lm}+\Gamma_{lm}$ \emph{exactly} as a sum of positive terms, and we do not need to rely on any lossy approximations or estimates. This new variable $\chi_{lm}$ is so suited for the problem that the liquid boundary condition $(\grad\cdot\bs\theta)|_{\partial B_R}=0$ manifests as simply the condition $\chi_{lm}'(R)=0$ and the boundary terms $\Gamma_{lm}$ can be converted back to the right order in terms of derivative count. This new variable $\chi_{lm}$ would also in fact work for the gaseous case, thus it would provide a simpler proof for the gaseous case in \cite{Jang_Makino_2020} and give an exact bound for the ``energy'' lower bound.

\begin{lemma}[Spherical harmonics decomposition]\label{Spherical harmonics decomposition}
Let $\bs\theta\in H^2(B_R)$. Then 
\begin{align}
g&:=\grad\cdot(\bar\rho\bs\theta)\in H^1(B_R)\\
\Psi(\mb x) &:= \frac1{4\pi} \K g(\mb x)\in H^2(\R^3)\cap C^1(\R^3),
\end{align}
and they 
can be expanded in spherical harmonics 
\begin{align}
g(\mb x)&=\sum_{l=0}^\infty\sum_{m=-l}^lg_{lm}(r)Y_{lm}(\mb x)\qquad\qquad\text{ on }\ B_R,\label{E:g expansion}\\
\Psi(\mb x) &= \sum_{l=0}^\infty\sum_{m=-l}^l \Psi_{lm}(r) Y_{lm}(\mb x)\qquad\qquad\text{ on }\ \R^3, \label{E:PSIDEF}
\end{align}
that converge in $L^2(B_R)$ and $L^2(\R^3)$ respectively, where the spherical harmonics $Y_{lm}$ are introduced in Appendix~\ref{A:SPHERICALHARMONICS}. 
Moreover, $\Psi_{lm}$ are related to $g_{lm}$ by
\begin{align}
\Psi_{lm}(r) &= {-1\over 2l+1}\brac{\int_0^{r}{y^{l+2}\over r^{l+1}}g_{lm}(y)\d y+\int_{r}^R{r^l\over y^{l-1}}g_{lm}(y)\d y}\label{g and Psi relation}\\
g_{lm}&=\lpc^{\<l\>}\Psi_{lm}:=\brac{{1\over r^2}\brac{r^2\Psi_{lm}'}'-{l(l+1)\over r^2}\Psi_{lm}}.\label{E:LAPLACEL}
\end{align}
With this, the following identity holds:
\begin{align}
\<\mb L\bs\theta,\bs\theta\>_{\bar\rho}
&=\sum_{l=0}^\infty\sum_{m=-l}^l(\Lambda_{lm}+\Gamma_{lm}),
\end{align}
where, for $l\ge0$ and $m\in\{-l,\dots,l\}$,
\begin{subequations}\label{E:LAMBDALMDEF}
\begin{align}
\Lambda_{lm}&: = \int_0^R\brac{\gamma\bar\rho^{\gamma-2}g_{lm}^2+4\pi g_{lm}\Psi_{lm}}r^2\d r\\
\Gamma_{lm}&:=-{R^2\over\bar\rho'(R)}\brac{\brac{\gamma+{4\pi R\over\bar\rho'(R)}{1\over 2l+1}}g_{lm}(R)^2+8\pi g_{lm}(R)\Psi_{lm}(R)}.
\end{align}
\end{subequations}
\end{lemma}
\begin{proof}
$\bs\theta\in H^2(B_R)$ immediately  gives that $g\in H^1(B_R)$. Since $\Psi$ is a convolution of $g$ with the kernel $|\ph|^{-1}$, where $g$ is trivially extended by 0 on $\mathbb R^3\setminus B_R$, standard computation shows $\Psi\in C^1(\R^3)\cap H^2(\R^3)$. Since spherical harmonics form an $L^2$ basis (see \cite{Atkinson_Han_2012, Jackson_1962, Courant_Hilbert_1953} and Appendix~\ref{A:SPHERICALHARMONICS}), we have the spherical harmonics expansion \eqref{E:g expansion}-\eqref{E:PSIDEF} for $g$ and $\Psi$ in $L^2$.

By Lemma \ref{potential decomposition} we have
\[{1\over|\mb x-\mb y|}=4\pi\sum_{l=0}^\infty\sum_{m=-l}^l{1\over 2l+1}{\min\{|\mb x|,|\mb y|\}^l\over\max\{|\mb x|,|\mb y|\}^{l+1}}Y_{lm}(\mb y)Y_{lm}(\mb x)\]
which converge uniformly on all compact set in $\{(\mb x,\mb y):|\mb x|\not=|\mb y|\}$. So we have
\begin{align}
\mathcal{K}g(\mb x)
&=-4\pi\sum_{l=0}^\infty\sum_{m=-l}^l{1\over 2l+1}Y_{lm}(\mb x)\brac{\int_{B_{|\mb x|}(\mb 0)}{|\mb y|^l\over|\mb x|^{l+1}}gY_{lm}\d\mb y+\int_{B_{|\mb x|}(\mb 0)^c}{|\mb x|^l\over|\mb y|^{l+1}}gY_{lm}\d\mb y}\nonumber\\
&=-4\pi\sum_{l=0}^\infty\sum_{m=-l}^l{1\over 2l+1}Y_{lm}(\mb x)\brac{\int^{|\mb x|}_0{y^{l+2}\over|\mb x|^{l+1}}g_{lm}\d y+\int_{|\mb x|}^R{|\mb x|^l\over y^{l-1}}g_{lm}\d y}\label{Kg formula}.
\end{align}
We therefore conclude that 
\[
\Psi_{lm}(r) = {-1\over 2l+1}\brac{\int_0^{r}{y^{l+2}\over r^{l+1}}g_{lm}(y)\d y+\int_{r}^R{r^l\over y^{l-1}}g_{lm}(y)\d y}
\]
since spherical harmonics expansion is unique (using standard Hilbert space theory and the fact that spherical harmonics form an $L^2$ basis for $L^2$ functions on the sphere). Inverting this expression, we get \eqref{E:LAPLACEL}.

Now using the spherical harmonics expansion for $g$ and $\Psi$, we get
\begin{align*}
\int_{\partial B_R}g\K_{\partial}g\;\d S
&=-\int_{\partial B_R}\int_{\partial B_R}{g(\mb x)g(\mb y)\over|\mb x-\mb y|}\d S(\mb x)\d S(\mb y)\\
&=-{4\pi\over R}\sum_{l=0}^\infty\sum_{m=-l}^l{1\over 2l+1}\int_{\partial B_R}\int_{\partial B_R}g(\mb x)g(\mb y)Y_{lm}(\mb x)Y_{lm}(\mb y)\d S(\mb x)\d S(\mb y)\\
&=-4\pi R^3\sum_{l=0}^\infty\sum_{m=-l}^l{g_{lm}(R)^2\over 2l+1}
\end{align*}
and
\begin{align*}
\<\mb L\bs\theta,\bs\theta\>_{\bar\rho}
&=\int_{B_R}g\brac{\gamma\bar\rho^{\gamma-2}+\K-\bar\rho'(R)^{-1}\K_{\partial}}g\;\d\mb x\\
&\quad-\int_{\partial B_R}\bar\rho'(R)^{-1}g\brac{\gamma+\K-\bar\rho'(R)^{-1}\K_{\partial}}g\;\d S\\
&=\int_{B_R}g\brac{\gamma\bar\rho^{\gamma-2}g+4\pi\Psi}\;\d\mb x
-\int_{\partial B_R}\bar\rho'(R)^{-1}g\brac{\gamma g+8\pi\Psi-\bar\rho'(R)^{-1}\K_{\partial}g}\;\d S\\
&=\int_{B_R}\brac{\gamma\bar\rho^{\gamma-2}|g|^2+4\pi g\Psi}\;\d\mb x
-{1\over\bar\rho'(R)}\int_{\partial B_R}\brac{\gamma|g|^2+8\pi g\Psi-\bar\rho'(R)^{-1}g\K_{\partial}g}\;\d S\\
&=\sum_{l=0}^\infty\sum_{m=-l}^l\int_0^r\brac{\gamma\bar\rho^{\gamma-2}g_{lm}^2+4\pi g_{lm}\Psi_{lm}}r^2\d r\\
&\quad -\sum_{l=0}^\infty\sum_{m=-l}^l{R^2\over\bar\rho'(R)}\brac{\gamma g_{lm}(R)^2+8\pi g_{lm}(R)\Psi_{lm}(R)+{4\pi R\over\bar\rho'(R)}{g_{lm}(R)^2\over 2l+1}}\\
&=\sum_{l=0}^\infty\sum_{m=-l}^l(\Lambda_{lm}+\Gamma_{lm}).\qedhere
\end{align*} 
\end{proof}

\begin{lemma}\label{Psi boundary condition}
Take the setting of Lemma \ref{Spherical harmonics decomposition}. We have
\begin{align*}
\Psi_{lm}'(r)&=-{l+1\over r}\Psi_{lm}(r)\qquad\text{ for }\qquad r\geq R.
\end{align*}
\begin{align*}
\int_R^\infty\Psi_{lm}'(r)^2r^2\d r
=(l+1)^2\int_R^\infty\Psi_{lm}(r)^2\d r
={(l+1)^2R\over 2l+1}\Psi_{lm}(R)^2.
\end{align*}
\end{lemma}
\begin{proof}
Differentiating \eqref{g and Psi relation} we have
\begin{align}\label{derivative of Psi_lm}
\Psi_{lm}'(r) &= {-1\over 2l+1}\brac{rg_{lm}(r)-rg_{lm}(r)-(l+1)\int_0^{r}{y^{l+2}\over r^{l+2}}g_{lm}(y)\d y+l\int_{r}^R{r^{l-1}\over y^{l-1}}g_{lm}(y)\d y}
\end{align}
and so
\begin{align*}
\Psi_{lm}'(r)&={l+1\over 2l+1}\int_0^{R}{y^{l+2}\over r^{l+2}}g_{lm}(y)\d y\qquad\text{ for }\qquad r\geq R.
\end{align*}
Since
\begin{align*}
\Psi_{lm}(r)&={-1\over 2l+1}\int_0^R{y^{l+2}\over r^{l+1}}g_{lm}(y)\d y\qquad\text{ for }\qquad r\geq R,
\end{align*}
we have
\begin{align*}
\Psi_{lm}'(r)&=-{l+1\over r}\Psi_{lm}(r)\qquad\text{ for }\qquad r\geq R
\end{align*}
\begin{align*}
\Psi_{lm}(r)&={R^{l+1}\over r^{l+1}}\Psi_{lm}(R)\qquad\text{ for }\qquad r\geq R.
\end{align*}
So we have
\begin{align*}
\int_R^\infty\Psi_{lm}(r)^2\d r=\int_R^\infty{R^{2l+2}\over r^{2l+2}}\Psi_{lm}(R)^2\d r
=-{R^{2l+2}\over 2l+1}\Psi_{lm}(R)^2\sbrac{1\over r^{2l+1}}_R^\infty
={R\over 2l+1}\Psi_{lm}(R)^2.\qquad\qedhere
\end{align*}
\end{proof}

From~\cite{Lam_2024} we have the following Theorem.

\begin{theorem}\label{radial positivity}
Let $\bar\rho$ be a liquid LE profile with adiabatic index $\gamma$ and radius $R$, and
\begin{align*}
\mathcal{L}\chi&:=-\gamma\partial_r\brac{\bar\rho^\gamma r^4\partial_r\chi}+(4-3\gamma)r^3\chi\partial_r\bar\rho^\gamma.
\end{align*}
Let $E:=\{\chi:\|\chi\|_E<\infty,3\chi(R)+R\chi'(R)=0\}$, where
\begin{align*}
\|\chi\|_E&:=\|r^2\chi'\|_{L^2([0,R])}^2+\|r^2\chi\|_{L^2([0,R])}^2+(R^3+R^4)|\chi(R)|^2.
\end{align*}
Then
\begin{enumerate}[i.]
\item When either $\gamma\geq 4/3$ or $\bar\rho(0)$ is close to 1, we have
\begin{align*}
\inf_{\|\chi\|_E>0}{\<\mathcal{L}\chi,\chi\>_{L^2([0,R])}\over\|\chi\|_E^2}>0.
\end{align*}
\item When $\gamma<4/3$ and $\bar\rho(0)$ is large, we have
\begin{align*}
\inf_{\|\chi\|_E>0}{\<\mathcal{L}\chi,\chi\>_{L^2([0,R])}\over\|\chi\|_E^2}<0.
\end{align*}
\end{enumerate}
\end{theorem}

We shall use Theorem~\ref{radial positivity} to obtain coercivity for the quadratic form $\Lambda_{00}$.

\begin{proposition}[$l=0$ mode bound]\label{l=0 mode bound}
Suppose $\bs\theta\in H^2(B_R)$ and $(\grad\cdot\bs\theta)|_{\partial B_R}=0$. When case i. of Lemma \ref{radial positivity} holds we have
\begin{align}\label{E:ZEROZEROBOUND}
\Lambda_{00}+\Gamma_{00} \gtrsim  \int_0^R g_{00}^2r^2\d r+{R^2\over\bar\rho'(R)^2}g_{00}(R)^2,
\end{align}
where $\Lambda_{00},\Gamma_{00},g_{00}$ are as defined in Lemma \ref{Spherical harmonics decomposition}.
\end{proposition}
\begin{proof}
$\bs\theta\in H^2(B_R)$ means that $\bar\rho\bs\theta$ is well defined on $\partial B_R$ (trace theorem). Since $(\grad\cdot\bs\theta)|_{\partial B_R}=0$ and $\bar\rho\bs\theta=\grad\Psi+\mb C$ where $\mb C$ is divergence-free, we have
\[\int_{\partial B_R}\partial_r\Psi\;\d S=\int_{\partial B_R}\grad\Psi\cdot\d\mb S=\int_{\partial B_R}\bar\rho\bs\theta\cdot\d\mb S={1\over\bar\rho'(R)}\int_{\partial B_R}g\;\d S.\]
It follows that $\Psi_{00}'(R)=\bar\rho'(R)^{-1}g_{00}(R)$. Now taking the derivative of \eqref{g and Psi relation} and using $g_{00}(r)=0$ for $r>R$, we see that in fact we must have
\begin{align}
\Psi_{00}'(r)={R\over r\bar\rho'(R)}g_{00}(R)\qquad\text{for}\qquad r\geq R.\label{00 boundary condition second}
\end{align}

Denote
\begin{align}\label{E:VARPHIDEF}
\chi : = {\Psi_{00}'\over r\bar\rho},
\end{align}

Lemma \ref{Psi boundary condition} gives us
\begin{align*}
\Psi_{00}(R)=-R\Psi_{00}'(R).
\end{align*}
Using \eqref{E:LAPLACEL} and \eqref{00 boundary condition second}, we get 
\begin{align*}
\Gamma_{00}&=-{R^2\over\bar\rho'(R)}\brac{\brac{\gamma+{4\pi R\over\bar\rho'(R)}}g_{00}(R)^2+8\pi g_{00}(R)\Psi_{00}(R)}\\
&=-{R^2\over\bar\rho'(R)}\brac{\gamma g_{00}(R)^2+4\pi g_{00}(R)\Psi_{00}(R)}
=-\gamma R^2\bar\rho'(R)\Psi_{00}'(R)^2-4\pi R^2\Psi_{00}'(R)\Psi_{00}(R)
\end{align*}
\begin{align*}
\Lambda_{00}&=\int_0^R\brac{\gamma\bar\rho^{\gamma-2}g_{00}^2+4\pi g_{00}\Psi_{00}}r^2\d r
=\int_0^R\brac{{\gamma\over r^2}\bar\rho^{\gamma-2}\brac{\brac{r^2\Psi_{00}'}'}^2+4\pi\brac{r^2\Psi_{00}'}'\Psi_{00}}\d r\\
&=\int_0^R\brac{{\gamma\over r^2}\bar\rho^{\gamma-2}\brac{\brac{r^2\Psi_{00}'}'}^2-4\pi r^2(\Psi_{00}')^2}\d r+4\pi R^2\Psi_{00}'(R)\Psi_{00}(R)\\
&=\int_0^R\brac{{\gamma\over r^2}\bar\rho^{\gamma-2}\brac{\brac{r^3\bar\rho\chi}'}^2-4\pi\chi^2\bar\rho^2r^4}\d r-\Gamma_{00}-\gamma R^4\bar\rho'(R)\chi(R)^2.
\end{align*}
Now since $0={\gamma\over\gamma-1}\lpc\bar\rho^{\gamma-1}+4\pi\bar\rho$ as in \eqref{E:equation for bar-w}, we see that 
\begin{align}
&\Lambda_{00}+\Gamma_{00}\nonumber\\
&=\int_0^R\brac{{\gamma\over r^2}\bar\rho^{\gamma-2}\brac{\brac{r^3\bar\rho\chi}'}^2+{\gamma\over\gamma-1}r^4\chi^2\bar\rho\lpc\bar\rho^{\gamma-1}}\d r-\gamma R^4\bar\rho'(R)\chi(R)^2\nonumber\\
&=\int_0^R\bigg({\gamma\over r^2}\bar\rho^{\gamma-2}\brac{3r^2\bar\rho\chi+r^3\bar\rho'\chi+r^3\bar\rho\chi'}^2
+{\gamma\over\gamma-1}(r^2(\bar\rho^{\gamma-1})')'\chi^2\bar\rho r^2\bigg)\d r-\gamma R^4\bar\rho'(R)\chi(R)^2\nonumber\\
&=\int_0^R\bigg(\gamma\brac{3r\bar\rho^{\gamma/2}\chi+r^2\bar\rho^{\gamma/2-1}\bar\rho'\chi+r^2\bar\rho^{\gamma/2}\chi'}^2
-{\gamma\over\gamma-1}r^2(\bar\rho^{\gamma-1})'(\chi^2\bar\rho r^2)'\bigg)\d r\nonumber\\
&\quad\underbrace{+{\gamma\over\gamma-1}R^4(\bar\rho^{\gamma-1})'(R)\chi(R)^2-\gamma R^4\bar\rho'(R)\chi(R)^2}_{=0}\nonumber\\
&=\int_0^R\bigg(\gamma\big(9r^2\bar\rho^\gamma\chi^2+r^4\bar\rho^{\gamma-2}(\bar\rho')^2\chi^2+r^4\bar\rho^\gamma(\chi')^2
+6r^3\bar\rho^{\gamma-1}\bar\rho'\chi^2+2r^4\bar\rho^{\gamma-1}\bar\rho'\chi\chi'+6r^3\bar\rho^{\gamma}\chi\chi'\big)\nonumber\\
&\qquad\qquad-\gamma r^2\bar\rho^{\gamma-2}\bar\rho'(2\chi\chi'\bar\rho r^2+\chi^2\bar\rho'r^2+2r\chi^2\bar\rho)\bigg)\d r\nonumber\\
&=\int_0^R\gamma\brac{9r^2\bar\rho^\gamma\chi^2+r^4\bar\rho^\gamma(\chi')^2+4\gamma^{-1}r^3(\bar\rho^\gamma)'\chi^2+6r^3\bar\rho^\gamma\chi\chi'}
\d r\nonumber\\
&=\int_0^R\gamma\brac{r^4\bar\rho^\gamma(\chi')^2+(4\gamma^{-1}-3)r^3(\bar\rho^\gamma)'\chi^2}
\d r+3\gamma R^3\chi(R)^2\nonumber\\
&=\int_0^R\brac{\gamma r^4\bar\rho^\gamma(\chi')^2+(4-3\gamma)r^3(\bar\rho^\gamma)'\chi^2}
\d r+3\gamma R^3\chi(R)^2
=\<\mathcal{L}\chi,\chi\>_{L^2([0,R])}.\label{00 mode in chi form}
\end{align}
By case i of Lemma \ref{radial positivity} we then get
\begin{align*}
\Lambda_{00}+\Gamma_{00}&\gtrsim\|\chi\|_{E}^2.
\end{align*}
Then for $\epsilon$ small enough, we get
\begin{align*}
(1+\epsilon)(\Lambda_{00}+\Gamma_{00})&\geq\int_0^R\epsilon\brac{\gamma\bar\rho^{\gamma-2}g_{00}^2-4\pi r^2(\Psi_{00}')^2}\d r-\epsilon\gamma R^2\bar\rho'(R)\Psi_{00}'(R)^2+C\|\chi\|_{E}^2.
\end{align*}
Finally, using \eqref{E:VARPHIDEF} and choosing $\epsilon$ small enough we get
\begin{align*}
\Lambda_{00}+\Gamma_{00} \gtrsim  \int_0^R\brac{\bar\rho^{\gamma-2}g_{00}^2+(\Psi_{00}')^2}r^2\d r+R^2\Psi_{00}'(R)^2.
\end{align*}
By \eqref{00 boundary condition second} we can convert the boundary term into $g_{00}(R)$ to get~\eqref{E:ZEROZEROBOUND}.
\end{proof}

To prove the positivity of the higher modes, we will reformulate our functional into variables under which positivity is apparent.

\begin{lemma}[$\chi_{lm}$ formulation]\label{chi_lm formulation}
Take the setting of Lemma \ref{Spherical harmonics decomposition}. Suppose $\bs\theta\in H^2(B_R)$ and $(\grad\cdot\bs\theta)|_{\partial B_R}=0$, then we can write $g_{lm}=(\bar\rho\chi_{lm})'r^{l-1}$, where $\chi_{lm}$ is such that $\chi_{lm}'(R)=0$. And we have that
\begin{align}
  \Lambda_{lm}+\Gamma_{lm} 
  &=\int_0^R\brac{\gamma \bar\rho^{\gamma-2}(\bar\rho\chi_{lm})'^2-4\pi(\bar\rho\chi_{lm})^2}r^{2l}\d r-\gamma\bar\rho'(R)\chi_{lm}(R)^2R^{2l}\label{chi_lm energy one}\\
  &=\int_0^R\brac{\gamma\bar\rho^\gamma\chi_{lm}'^2r^{2l} -2(l-1)(\bar\rho^{\gamma})'\chi_{lm}^2 r^{2l-1}}\d r.\label{chi_lm energy two}
\end{align}
\end{lemma}
\begin{proof}
We can decompose $\bs\theta$ in vector spherical harmonics (see \eqref{vector spherical harmonics}):
\[\bs\theta=\sum_{l=0}^\infty\sum_{m=-l}^l\brac{\Theta^{[0]}_{lm}\mb Y^{[0]}_{lm}+\Theta^{[1]}_{lm}\mb Y^{[1]}_{lm}+\Theta^{[2]}_{lm}\mb Y^{[2]}_{lm}}.\]
Using \eqref{lpc of spherical harmonics} and the fact that $\grad\cdot(\mb x\times\grad)=0$ we can compute the divergence of $\bs\theta$ and $\bar\rho\bs\theta$ in terms of $\Theta^{[k]}_{lm}$. The condition $\grad\cdot\bs\theta|_{\partial B_R}=0$ means
\begin{align*}
  0
  &=\eva{\sum_{l=0}^\infty\sum_{m=-l}^l\brac{r^{-2}(\Theta^{[0]}_{lm}r^2)' +\Theta^{[1]}_{lm}r\lpc }Y_{lm}}_{r=R}
\end{align*}
which means for all $l,m$ we have
\begin{align}
  0 &=\eva{\brac{r^{-2}(\Theta^{[0]}_{lm}r^2)' -l(l+1)\Theta^{[1]}_{lm}r^{-1}}}_{r=R}\nonumber\\
  &=R^{-1}\brac{2\Theta^{[0]}_{lm}(R)-l(l+1)\Theta^{[1]}_{lm}(R)}+\brac{\Theta^{[0]}_{lm}}'(R).\label{boundary divergence condition}
\end{align}
Now
\begin{align*}
g &= \grad\cdot(\bar\rho\bs\theta)
=\sum_{l=0}^\infty\sum_{m=-l}^l\brac{r^{-2}(\bar\rho\Theta^{[0]}_{lm}r^2)' +\bar\rho\Theta^{[1]}_{lm}r\lpc }Y_{lm}\\
&=\sum_{l=0}^\infty\sum_{m=-l}^l\brac{r^{-2}(\bar\rho\Theta^{[0]}_{lm}r^2)' -l(l+1)\bar\rho\Theta^{[1]}_{lm}r^{-1}}Y_{lm}\\
&=\sum_{l=0}^\infty\sum_{m=-l}^l\brac{r^{-2}(\bar\rho\Theta^{[0]}_{lm}r^{1-l}r^{l+1})' -l(l+1)\bar\rho\Theta^{[1]}_{lm}r^{-1}}Y_{lm}\\
&=\sum_{l=0}^\infty\sum_{m=-l}^l\brac{(\bar\rho\Theta^{[0]}_{lm}r^{1-l})'r^{l-1} +(l+1)\brac{\Theta^{[0]}_{lm}-l\Theta^{[1]}_{lm}}\bar\rho r^{-1}}Y_{lm}\\
&=\sum_{l=0}^\infty\sum_{m=-l}^l\brac{(\bar\rho\Theta^{[0]}_{lm}r^{1-l})'r^{l-1} -(l+1)\brac{\bar\rho\bar\rho^{-1}\int^R_r\brac{\Theta^{[0]}_{lm}-l\Theta^{[1]}_{lm}}(y)\bar\rho(y) y^{-l}\d y}'r^{l-1}}Y_{lm}\\
&=\sum_{l=0}^\infty\sum_{m=-l}^l(\bar\rho\chi_{lm})'r^{l-1}Y_{lm}
\end{align*}
where
\begin{align}
  \chi_{lm} &= \Theta^{[0]}_{lm}r^{1-l}-(l+1)\bar\rho^{-1}\int^R_r\brac{\Theta^{[0]}_{lm}-l\Theta^{[1]}_{lm}}(y)\bar\rho(y) y^{-l}\d y.\label{chi_lm def}
\end{align}
Since $\bs\theta\in H^2(B_R)$, as $r\to 0$ we must have $\Theta^{[0]}_{00}=O(r)$ and for $l\geq 1$, $\Theta^{[0]}_{lm}=\Theta^{[1]}_{lm}=O(r^{l-1})$ and $\Theta^{[2]}_{lm}=O(r^l)$. This means $\chi_{00}=O(1)$ and $\chi_{lm}=O(\ln r)$ as $r\to 0$ for $l\geq 1$. These are enough to ensure that the boundary terms at 0 for the integration by parts in what follows vanish.

Using \eqref{boundary divergence condition} we have that
\begin{align*}
  \chi_{lm}'(R) &= (\Theta^{[0]}_{lm}r^{1-l})'(R)+(l+1)\brac{\Theta^{[0]}_{lm}(R)-l\Theta^{[1]}_{lm}(R)} R^{-l}\\
  &= (\Theta^{[0]}_{lm})'(R)R^{1-l}+\brac{2\Theta^{[0]}_{lm}(R)-l(l+1)\Theta^{[1]}_{lm}(R)} R^{-l}=0.
\end{align*}

We have, using the notation of Lemma \ref{Spherical harmonics decomposition},
\begin{align*}
g_{lm} &=(\bar\rho\chi_{lm})'r^{l-1}\\
  \Psi_{lm} &= -{1\over 2l+1}\brac{\int_0^r{y^{2l+1}\over r^{l+1}}(\bar\rho\chi_{lm})'(y)\d y+\int_r^R r^l(\bar\rho\chi_{lm})'(y)\d y}Y_{lm}\\
  &= -{1\over 2l+1}\brac{r^{l}(\bar\rho\chi_{lm})(r)-(2l+1)\int_0^r{y^{2l}\over r^{l+1}}(\bar\rho\chi_{lm})(y)\d y+r^l(\chi_{lm}(R)-(\bar\rho\chi_{lm})(r))}Y_{lm}\\
  &=\brac{\int_0^r{y^{2l}\over r^{l+1}}(\bar\rho\chi_{lm})(y)\d y-{r^l\chi_{lm}(R)\over 2l+1}}Y_{lm}
  =Y_{lm}\int_0^r{y^{2l}\over r^{l+1}}((\bar\rho\chi_{lm})(y)-\chi_{lm}(R))\d y.
\end{align*}
So we have
\begin{align*}
  \Lambda_{lm} &= \int_0^R\brac{\gamma \bar\rho^{\gamma-2}(\bar\rho\chi_{lm})'^2r^{2l}+4\pi(\bar\rho\chi_{lm})' \brac{\int_0^r y^{2l}((\bar\rho\chi_{lm})(y)-\chi_{lm}(R))\d y}}\d r\\
  &=4\pi\chi_{lm}(R)\brac{\int_0^R y^{2l}((\bar\rho\chi_{lm})(y)-\chi_{lm}(R))\d y}\\
  &\quad+\int_0^R \brac{\gamma\bar\rho^{\gamma-2}(\bar\rho\chi_{lm})'^2r^{2l}-4\pi r^{2l}\bar\rho\chi_{lm}(\bar\rho\chi_{lm}-\chi_{lm}(R))}\d r\\
  &=\int_0^R\brac{\gamma\bar\rho^{\gamma-2}(\bar\rho\chi_{lm})'^2r^{2l}-4\pi r^{2l}(\bar\rho\chi_{lm}-\chi_{lm}(R))^2}\d r\\
  \Gamma_{lm} &= -R^2\brac{\brac{\gamma\bar\rho_{lm}'(R)+{4\pi R\over 2l+1}}\chi_{lm}(R)^2R^{2l-2}+8\pi\chi_{lm}(R)R^{-2}\int_0^R y^{2l}((\bar\rho\chi_{lm})(y)-\chi_{lm}(R))\d y}\\
  &=\brac{-\gamma\bar\rho'(R)-{4\pi R\over 2l+1}}\chi_{lm}(R)^2R^{2l}-8\pi\chi_{lm}(R)\int_0^R r^{2l}(\bar\rho\chi_{lm}-\chi_{lm}(R))\d r.
\end{align*}
Therefore we have
\begin{align*}
\Lambda_{lm}+\Gamma_{lm}
  &= \int_0^R\brac{\gamma \bar\rho^{\gamma-2}(\bar\rho\chi_{lm})'^2-4\pi((\bar\rho\chi_{lm}-\chi_{lm}(R)+\chi_{lm}(R))^2-\chi_{lm}(R)^2)}r^{2l}\d r\\
  &\quad+\brac{-\gamma\bar\rho'(R)-{4\pi R\over 2l+1}}\chi_{lm}(R)^2R^{2l}\\
  &=\int_0^R\brac{\gamma \bar\rho^{\gamma-2}(\bar\rho\chi_{lm})'^2-4\pi(\bar\rho\chi_{lm})^2}r^{2l}\d r-\gamma\bar\rho'(R)\chi_{lm}(R)^2R^{2l}.
\end{align*}
By \eqref{LE ODE} we have $0={\gamma\over\gamma-1}(r^2(\bar\rho^{\gamma-1})')'+4\pi r^2\bar\rho=\gamma(r^2\bar\rho^{\gamma-2}\bar\rho')'+4\pi r^2\bar\rho$, so we get
\begin{align*}
  \Lambda_{lm}+\Gamma_{lm} &= \int_0^R\brac{\gamma \bar\rho^{\gamma-2}(\bar\rho\chi_{lm})'^2 r^{2l}+\gamma(r^2\bar\rho^{\gamma-2}\bar\rho')'\bar\rho\chi_{lm}^2r^{2l-2}}\d r-\gamma\bar\rho'(R)\chi_{lm}(R)^2R^{2l}\\
  &= \int_0^R\brac{\gamma\bar\rho^{\gamma-2}(\bar\rho\chi_{lm})'^2 r^{2l}-\gamma r^2\bar\rho^{\gamma-2}\bar\rho'(\bar\rho\chi_{lm}^2r^{2l-2})'}\d r\\
  &=\gamma\int_0^R\brac{\bar\rho^\gamma\chi_{lm}'^2r^{2l} -2(l-1)\bar\rho^{\gamma-2}\bar\rho'\bar\rho\chi_{lm}^2 r^{2l-1}}\d r\\
  &=\int_0^R\brac{\gamma\bar\rho^\gamma\chi_{lm}'^2r^{2l} -2(l-1)(\bar\rho^{\gamma})'\chi_{lm}^2 r^{2l-1}}\d r.\qedhere
\end{align*}
\end{proof}

\begin{remark}
The $\chi_{lm}$ here in Lemma \ref{chi_lm formulation} is not the spherical harmonics expansion of $\chi$ in the proof of Proposition \ref{l=0 mode bound}, nor is $\chi_{00}$ equal to $\chi$. At first sight, it might appear that for $l=0$, \eqref{chi_lm energy two} contradicts the result of non-negativity in Proposition \ref{l=0 mode bound} as choosing $\chi_{00}=1$ in \eqref{chi_lm energy two} gives a strictly negative number. However, $\chi_{00}$ cannot be chosen freely as it is constrained by \eqref{chi_lm def} since it arises from $\bs\theta$. To illustrate, we will show below that no $\bs\theta\in H^2(B_R)$ with $(\grad\cdot\bs\theta)|_{\partial B_R}=0$ can give rise to $\chi_{00}=1$. Suppose that $\chi_{00}=1$, this means
\begin{align*}
  1 &= \Theta^{[0]}_{00}r-\bar\rho^{-1}\int^R_r\Theta^{[0]}_{00}(y)\bar\rho(y)\d y.
\end{align*}
This means $(\bar\rho(1-\Theta^{[0]}_{00}r))'=\Theta^{[0]}_{00}\bar\rho$ which we can rewrite as
\begin{align*}
\bar\rho'=2\Theta^{[0]}_{00}\bar\rho+(\bar\rho\Theta^{[0]}_{00})'r=r^{-1}(\bar\rho\Theta^{[0]}_{00}r^2)'.
\end{align*}
Therefore,
\begin{align*}
\int_0^ry\bar\rho'(y)\d y=\bar\rho\Theta^{[0]}_{00}r^2+C.
\end{align*}
From \eqref{boundary divergence condition} we have
\begin{align*}
  0 &= 2R^{-1}\Theta^{[0]}_{00}(R)+\brac{\Theta^{[0]}_{00}}'(R).
\end{align*}
This means $C(\bar\rho^{-1})'(R)=\brac{\bar\rho^{-1}\int_0^ry\bar\rho'(y)\d y}'(R)$ and that
\begin{align*}
C = \int_0^Ry\bar\rho'(y)\d y - R = \int_0^R\bar\rho\,\d r\not=0.
\end{align*}
But this means $\Theta^{[0]}_{00}\not=O(r)$ which contradicts $\bs\theta\in H^2(B_R)$. Thus, whilst $\chi_{lm}$ is very useful for showing the non-negativity of the $l\geq 1$ modes, it is not very suitable for studying the radial $l=0$ mode. Hence why in Proposition \ref{l=0 mode bound} we used the different variable $\chi$.
\end{remark}

From \eqref{chi_lm energy one} we immediately see that

\begin{corollary}[Bound from above]\label{Bound from above}
Take the setting of Lemma \ref{Spherical harmonics decomposition}. Suppose $\bs\theta\in H^2(B_R)$ and $(\grad\cdot\bs\theta)|_{\partial B_R}=0$, then for all $l,m$ we have
\begin{align*}
\Lambda_{lm}+\Gamma_{lm}\lesssim\int_0^Rg_{lm}^2r^2\d r-{R^2\over\bar\rho'(R)}g_{lm}(R)^2.
\end{align*}
\end{corollary}

\begin{proposition}[$l=1$ modes bound]\label{l=1 modes bound}
Suppose $\bs\theta\in H^2(B_R)$ and $(\grad\cdot\bs\theta)|_{\partial B_R}=0$. Then
\begin{align*}
\Lambda_{1m}+\Gamma_{1m} \gtrsim  \int_0^R g_{1m}^2r^2\d r+{R^2\over\bar\rho'(R)^2}g_{1m}(R)^2\qquad\text{ when }\qquad \<\bs\theta,\mb e_i\>_{\bar\rho}=0\text{ for all }i=1,2,3,
\end{align*}
where $\Lambda_{1m},\Gamma_{1m},g_{1m}$ are as defined in Lemma \ref{Spherical harmonics decomposition}.
\end{proposition}
\begin{proof}
The orthogonality condition
\begin{align*}
0&=\<\bs\theta,\mb e_i\>_{\bar\rho}=\int_{B_R}\bar\rho\bs\theta\cdot\grad x^i\;\d\mb x
=-\int_{B_R}gx^i\;\d\mb x+\int_{\partial B_R} x^i\bs\theta\cdot\d\mb S\\
&=-\int_{B_R}gx^i\;\d\mb x+{1\over\bar\rho'(R)}\int_{\partial B_R} x^ig\;\d S
\end{align*}
means that 
$0=-\int_0^Rg_{1m}r^3\d r+R^3\bar\rho'(R)^{-1}g_{1m}(R)$
and therefore
\begin{align*}
  0=-\int_0^R(\bar\rho\chi_{1m})'r^3\d r+R^3\chi_{1m}(R)
  =3\int_0^R\bar\rho\chi_{1m}r^2\d r
\end{align*}
where $\chi_{1m}$ is as defined in Lemma \ref{chi_lm formulation}. This constraint eliminates non-zero constant functions, hence we can apply Poincaré inequality to get that
\begin{align*}
\int_0^R\chi_{1m}^2r^2\d r\lesssim\int_0^R\chi_{1m}'^2r^2\d r.
\end{align*}
Therefore, from Lemma \ref{chi_lm formulation},
\begin{align*}
  \Lambda_{1m}+\Gamma_{1m} &=\int_0^R\gamma\bar\rho^\gamma\chi_{1m}'^2r^2\d r
  \gtrsim \int_0^R\brac{\gamma\bar\rho^\gamma\chi_{1m}'^2r^2+\chi_{1m}^2r^2}\d r\\
  &\gtrsim \int_0^R\brac{\gamma\bar\rho^\gamma\chi_{1m}'^2r^2+\chi_{1m}^2r^2}\d r +R^2\chi_{1m}(R)^2
\end{align*}
where the last line is by trace. Since $g_{1m}=(\bar\rho\chi_{1m})'=\bar\rho'\chi_{1m}+\bar\rho\chi_{1m}'$ and $g_{1m}(R)=\bar\rho'(R)\chi_{1m}(R)$ we get
\begin{align*}
\Lambda_{1m}+\Gamma_{1m} &\gtrsim\int_0^R g_{1m}^2r^2\d r+{R^2\over\bar\rho'(R)^2}g_{1m}(R)^2. \qedhere
\end{align*}
\end{proof}

\begin{proposition}[$l\geq 2$ modes bound]\label{l geq 2 modes bound}
Suppose $\bs\theta\in H^2(B_R)$ and $(\grad\cdot\bs\theta)|_{\partial B_R}=0$. Then for all $l\geq 2$ we have
\begin{align*}
\Lambda_{lm}+\Gamma_{lm} \gtrsim  \int_0^R g_{lm}^2r^2\d r+{R^2\over\bar\rho'(R)^2}g_{lm}(R)^2,
\end{align*}
where $\Lambda_{lm},\Gamma_{lm},g_{lm}$ are as defined in Lemma \ref{Spherical harmonics decomposition}.
\end{proposition}
\begin{proof}
From Lemma \ref{chi_lm formulation}
\begin{align*}
  \Lambda_{lm}+\Gamma_{lm} &=\int_0^R\brac{\gamma\bar\rho^\gamma\chi_{lm}'^2r^{2l} -2(l-1)(\bar\rho^{\gamma})'\chi_{lm}^2 r^{2l-1}}\d r\\
  &\gtrsim \int_0^R\brac{\chi_{lm}'^2r^{2l} +(l-1)\chi_{lm}^2 r^{2l}}\d r +R^{2l}\chi_{lm}(R)^2
\end{align*}
where the last line is by lemma \ref{Lane--Emden profile derivative prop} and trace. Since $g_{lm}=(\bar\rho\chi_{lm})'r^{l-1}=\bar\rho'\chi_{lm}r^{l-1}+\bar\rho\chi_{lm}'r^{l-1}$ and $g_{lm}(R)=\bar\rho'(R)\chi_{lm}(R)R^{l-1}$ we then have the desired bound.
\end{proof}

Finally, collecting all these results, we can prove Theorem \ref{linear operator coercivity}.

\begin{proof}[Proof of Theorem~\ref{linear operator coercivity}]
Corollary \ref{Bound from above} gives us
\begin{align*}
\<\mb L\bs\theta,\bs\theta\>_{\bar\rho}
&=\sum_{l=0}^\infty\sum_{m=-l}^l\brac{\Lambda_{lm}+\Gamma_{lm}}\\
&\lesssim\sum_{l=0}^\infty\sum_{m=-l}^l\brac{\int_0^R g_{lm}^2r^2\d r +{R^2\over|\bar\rho'(R)|}g_{lm}(R)^2}
=\|g\|_{L^2(B_R)}^2+{1\over|\bar\rho'(R)|}\|g\|_{L^2(\partial B_R)}^2.
\end{align*}
\begin{enumerate}[i.]
\item Note that since $\mb L\mb e_i=\mb 0$ and $\mb L$ symmetric, we have
\begin{align*}
\<\mb L\bs\theta,\bs\theta\>_{\bar\rho}
&=\inner{\mb L\brac{\bs\theta-\sum_{i=1}^3{\<\bs\theta,\mb e_i\>_{\bar\rho}\over\|\mb e_i\|_{\bar\rho}^2}\mb e_i},\brac{\bs\theta-\sum_{i=1}^3{\<\bs\theta,\mb e_i\>_{\bar\rho}\over\|\mb e_i\|_{\bar\rho}^2}\mb e_i}}_{\bar\rho}.
\end{align*}
So for i., it suffices to prove the case when $\<\bs\theta,\mb e_i\>_{\bar\rho}=0$, $i=1,2,3$.

Combining all the bounds we have for each $l,m$ from Propositions~\ref{l=0 mode bound}, \ref{l=1 modes bound} and \ref{l geq 2 modes bound}, we have
\begin{align*}
\<\mb L\bs\theta,\bs\theta\>_{\bar\rho}
&=\sum_{l=0}^\infty\sum_{m=-l}^l\brac{\Lambda_{lm}+\Gamma_{lm}}\\
&\gtrsim\sum_{l=0}^\infty\sum_{m=-l}^l\brac{\int_0^R g_{lm}^2r^2\d r +{R^2\over\bar\rho'(R)^2}g_{lm}(R)^2}
=\|g\|_{L^2(B_R)}^2+{1\over\bar\rho'(R)^2}\|g\|_{L^2(\partial B_R)}^2.
\end{align*}
\item By ii.~of Theorem \ref{radial positivity} there exist $\chi$ such that $\<\mathcal{L}\chi,\chi\>_{L^2([0,R])}<0$. Let $\bs\theta=(4\pi)^{-1/2}r\chi\mb e_r$, then in the notation of Lemma \ref{Spherical harmonics decomposition} we have $g=(4\pi)^{-1/2}r^{-2}(r^3\bar\rho\chi)'=r^{-2}(r^3\bar\rho\chi)'Y_{00}$ and $\Psi=\Psi_{00}Y_{00}$ so that
\begin{align*}
\<\mb L\bs\theta,\bs\theta\>_{\bar\rho}=\Lambda_{00}+\Gamma_{00}.
\end{align*}
By \eqref{derivative of Psi_lm} we have
\begin{align*}
\Psi_{00}'(r)=\int_0^r{y^2\over r^2}y^{-2}(y^3\bar\rho\chi)'\d y=r\bar\rho\chi.
\end{align*}
So by \eqref{00 mode in chi form},
\begin{align*}
\<\mb L\bs\theta,\bs\theta\>_{\bar\rho}&=\Lambda_{00}+\Gamma_{00}=\<\mathcal{L}\chi,\chi\>_{L^2([0,R])}<0.\qedhere
\end{align*}
\end{enumerate}
\end{proof}

\section*{Acknowledgements}

The author is supported by NWO grants VI.Vidi.223.019 and OCENW.M20.194 and previously by the EPSRC studentship grant EP/R513143/1 when this research began. The author thanks Mahir Had\v{z}i\'c and Juhi Jang for helpful discussions, and is grateful to Had\v{z}i\'c for introducing him to this problem.

\appendix
\section{Appendix}

\subsection{Spherical harmonics}\label{A:SPHERICALHARMONICS}

Spherical harmonics have real as well as complex versions. For the definition and basic properties of the complex version, see \cite{Jackson_1962}. The relation between complex spherical harmonics $Y^m_l:S^2\to\C$ and real spherical harmonics $Y_{lm}:S^2\to\R$ is
\begin{align*}
Y^m_l=\begin{cases}{1\over\sqrt 2}(Y_{l,-m}-iY_{lm})& m<0\\
Y_{l0}& m=0\\
{(-1)^m\over\sqrt 2}(Y_{lm}+iY_{l,-m})& m>0.\end{cases}
\end{align*}
We also have the relation
$
(Y^m_l)^*=(-1)^mY^{-m}_l.
$
The zeroth and first order real spherical harmonics are given by
\begin{alignat*}{3}
Y_{0,0}(\mb x)&={1\over\sqrt{4\pi}},& \qquad\qquad \ 
Y_{1,-1}(\mb x)&=\sqrt{3\over 4\pi}{x^2\over|\mb x|},\\
Y_{1,0}(\mb x)&=\sqrt{3\over 4\pi}{x^3\over|\mb x|},& \qquad\qquad \
Y_{1,1}(\mb x)&=\sqrt{3\over 4\pi}{x^1\over|\mb x|}.
\end{alignat*}
The spherical harmonics satisfy
\begin{align}\label{lpc of spherical harmonics}
\lpc Y_{lm}=-l(l+1)r^{-2}Y_{lm},\qquad\qquad\qquad\lpc(r^lY_{lm})=0,
\end{align}
and the following orthonormal conditions
\begin{align*}
\int_{S^2}Y_{lm}Y_{l'm'}\d S=\delta_{ll'}\delta_{mm'}=\int_{S^2}Y_{m}^m(Y_{l'}^{m'})^*\d S
\end{align*}
and they form a basis for $L^2(S^2)$ \cite{Atkinson_Han_2012} so that, in particular, any function $g\in L^2(S^2)$ has a spherical harmonics expansion
\[g=\sum_{l=0}^\infty\sum_{m=-l}^l g_{lm}Y_{lm},\qquad\qquad g_{lm}\in\R\]
that converge in $L^2(S^2)$. More generally, a function $g\in L^2(B_R)$ has a spherical harmonics expansion in $L^2(B_R)$,
\begin{align}
g=\sum_{l=0}^\infty\sum_{m=-l}^l g_{lm}(r)Y_{lm},\qquad\qquad g_{lm}:[0,R]\to\R.\label{spherical harmonics expansion}
\end{align}
Indeed, since $L^2(B_R)=L^2([0,R];L^2(S^2),r^2)=L^2([0,R];L^2(\partial B_r))$, or in other words
\[\int_{B_R}|\ph|\ \d\mb x=\int_0^R\int_{\partial B_r}|\ph|\ \d S\d r,\]
$g|_{\partial B_r}$ must be in $L^2(\partial B_r)$ for almost every $r\in[0,R]$. So a spherical harmonics expansion exists for almost every $r$. Now
\begin{align*}
\norm{g-\sum_{l=0}^N\sum_{m=-l}^l g_{lm}Y_{lm}}_{L^2(B_R)}^2&=\int_0^R\norm{g-\sum_{l=0}^N\sum_{m=-l}^l g_{lm}Y_{lm}}_{L^2(\partial B_r)}^2\d r\\
&\to 0\qquad\text{as}\qquad N\to\infty
\end{align*}
by the dominated convergence theorem (where the dominating function is $4\|g\|_{L^2(\partial B_r)}^2$). Hence \eqref{spherical harmonics expansion} converge in $L^2(B_R)$. Similarly, functions in $L^2(\R^3)$ have a spherical harmonics expansion.

The following lemma allows us to expand gravitational potentials in spherical harmonics.

\begin{lemma}\label{potential decomposition}
For $\mb x,\mb y\in\R^3$ we have
\begin{align*}
{1\over|\mb x-\mb y|}&=4\pi\sum_{l=0}^\infty\sum_{m=-l}^l{1\over 2l+1}{\min\{|\mb x|,|\mb y|\}^l\over\max\{|\mb x|,|\mb y|\}^{l+1}}Y_{lm}(\mb y)Y_{lm}(\mb x)
\end{align*}
and this expression converge uniformly for $(\mb x,\mb y)$ in any compact set in $\{(\mb r,\mb r')\in\R^6:|\mb r|\not=|\mb r'|\}$.
\end{lemma}
\begin{proof}
From~\cite{Jackson_1962} we have
\begin{align*}
{1\over|\mb x-\mb y|}&=4\pi\sum_{l=0}^\infty\sum_{m=-l}^l{1\over 2l+1}{\min\{|\mb x|,|\mb y|\}^l\over\max\{|\mb x|,|\mb y|\}^{l+1}}Y^m_{l}(\mb y)^*Y^m_{l}(\mb x).
\end{align*}
One derivation of this formula is as follows. Assume $r'=|\mb r'|<|\mb r|=r$, otherwise swap $\mb r'$ and $\mb r$. By the law of cosines,
\[{\displaystyle {\frac {1}{|\mathbf {r} -\mathbf {r} ' |}}={\frac {1}{\sqrt {r^{2}+(r')^{2}-2rr'\cos \gamma }}}={\frac {1}{r{\sqrt {1+h^{2}-2h\cos \gamma }}}}\quad {\hbox{with}}\quad h:={\frac {r'}{r}}.}\]
We find here the generating function of the Legendre polynomials $P_{\ell }(\cos \gamma )$:
\begin{align}
{\displaystyle {\frac {1}{\sqrt {1+h^{2}-2h\cos \gamma }}}=\sum _{\ell =0}^{\infty }h^{\ell }P_{\ell }(\cos \gamma ).}\label{Generating function of the Legendre polynomials}
\end{align}
Use of the spherical harmonic addition theorem
\[{\displaystyle P_{\ell }(\cos \gamma )={\frac {4\pi }{2\ell +1}}\sum _{m=-\ell }^{\ell }(-1)^{m}Y_{\ell }^{-m}(\theta ,\varphi )Y_{\ell }^{m}(\theta ',\varphi ')}\]
gives our first formula.
Since $|P_\ell(\cos\gamma)|\leq 1$ for all $\ell$, the power series in \eqref{Generating function of the Legendre polynomials} has radius of convergence 1, and uniform convergence for any compact set in $B_1$. By the identity theorem for analytic functions, the equality of \eqref{Generating function of the Legendre polynomials} holds for $h<1$. We thus conclude that the expansion for $|\mb r-\mb r'|^{-1}$ converge uniformly on any compact set in $\{(\mb r,\mb r')\in\R^6:|\mb r|\not=|\mb r'|\}$.
Moreover, in real spherical harmonics,
\begin{align*}
{1\over|\mb x-\mb y|}
&=4\pi\sum_{l=0}^\infty\sum_{m=-l}^l{1\over 2l+1}{\min\{|\mb x|,|\mb y|\}^l\over\max\{|\mb x|,|\mb y|\}^{l+1}}Y^m_{l}(\mb y)^*Y^m_{l}(\mb x)\\
&=4\pi\sum_{l=0}^\infty\sum_{m=-l}^l{(-1)^m\over 2l+1}{\min\{|\mb x|,|\mb y|\}^l\over\max\{|\mb x|,|\mb y|\}^{l+1}}Y^{-m}_{l}(\mb y)Y^m_{l}(\mb x)\\
&=4\pi\sum_{l=0}^\infty\sum_{m=1}^l{1\over 2}{1\over 2l+1}{\min\{|\mb x|,|\mb y|\}^l\over\max\{|\mb x|,|\mb y|\}^{l+1}}(Y_{l,m}(\mb y)-iY_{l,-m}(\mb y))(Y_{lm}(\mb x)+iY_{l,-m}(\mb x))\\
&\quad+4\pi\sum_{l=0}^\infty{1\over 2l+1}{\min\{|\mb x|,|\mb y|\}^l\over\max\{|\mb x|,|\mb y|\}^{l+1}}Y_{l0}(\mb y)Y_{l0}(\mb x)\\
&\quad+4\pi\sum_{l=0}^\infty\sum_{m=-l}^{-1}{1\over 2}{1\over 2l+1}{\min\{|\mb x|,|\mb y|\}^l\over\max\{|\mb x|,|\mb y|\}^{l+1}}(Y_{l,-m}(\mb y)+iY_{lm}(\mb y))(Y_{l,-m}(\mb x)-iY_{lm}(\mb x))\\
&=4\pi\sum_{l=0}^\infty\sum_{m=-l}^l{1\over 2l+1}{\min\{|\mb x|,|\mb y|\}^l\over\max\{|\mb x|,|\mb y|\}^{l+1}}Y_{lm}(\mb y)Y_{lm}(\mb x).\qedhere
\end{align*}
\end{proof}

Define the vector spherical harmonics via
\begin{align}\label{vector spherical harmonics}
\mb Y^{[0]}_{lm}=Y_{lm}\mb e_r,\qquad\qquad
\mb Y^{[1]}_{lm}=r\grad Y_{lm},\qquad\qquad
\mb Y^{[2]}_{lm}=\mb x\times\grad Y_{lm}.
\end{align}
These vector spherical harmonics are orthogonal to each other and
\begin{align}\label{vector spherical harmonics size}
\int_{S^2}\big|\mb Y^{[0]}_{lm}\big|^2\d S =1,\qquad\qquad\qquad
\int_{S^2}\big|\mb Y^{[1]}_{lm}\big|^2\d S =
\int_{S^2}\big|\mb Y^{[2]}_{lm}\big|^2\d S =l(l+1).
\end{align}
We can expand $L^2$ vector fields $\bs\theta$ as
\begin{align*}
\bs\theta=\sum_{l=0}^\infty\sum_{m=-l}^l\brac{\Theta_{lm}^{[0]}\mb Y^{[0]}_{lm}+\Theta_{lm}^{[1]}\mb Y^{[1]}_{lm}+\Theta_{lm}^{[2]}\mb Y^{[2]}_{lm}}.
\end{align*}
Details about these vector spherical harmonics can be found in \cite{Barrera_Estevez_Giraldo_1985, Freeden_Schreiner_2022}.

\nocite{Lin_Zeng_2022}
\bibliographystyle{plain} 
\bibliography{references}

\end{document}